\documentclass[reqno]{amsart}

\usepackage{a4wide}
\usepackage{color}
\usepackage{mathrsfs}
\usepackage{mathtools}
\usepackage{tikz-cd}
\usepackage{amsmath}
\usepackage{amssymb}
\numberwithin{equation}{section}
\usepackage[colorlinks,citecolor=green,linkcolor=red]{hyperref}
\usepackage{hyphenat}
\usepackage[latin1]{inputenc}

\newcommand{\B}{\mathbb{B}}
\newcommand{\N}{\mathbb{N}}
\newcommand{\R}{\mathbb{R}}
\newcommand{\sfd}{{\sf d}}
\renewcommand{\d}{{\mathrm d}}
\newcommand{\restr}[1]{\lower3pt\hbox{\(|_{#1}\)}}

\newcommand{\nchi}{{\raise.3ex\hbox{\(\chi\)}}}
\newcommand{\fr}{\penalty-20\null\hfill\(\blacksquare\)}
\newcommand{\X}{{\rm X}}
\newcommand{\Y}{{\rm Y}}
\newcommand{\mm}{\mathfrak m}
\renewcommand{\P}{{\rm P}}

\newtheorem{theorem}{Theorem}[section]
\newtheorem{corollary}[theorem]{Corollary}
\newtheorem{lemma}[theorem]{Lemma}
\newtheorem{proposition}[theorem]{Proposition}
\newtheorem{definition}[theorem]{Definition}
\newtheorem{example}[theorem]{Example}
\newtheorem{remark}[theorem]{Remark}

\title{Direct and inverse limits of normed modules}
\author{Enrico Pasqualetto}
\address{Department of Mathematics and Statistics,
P.O.\ Box 35 (MaD), FI-40014 University of Jyv\"{a}skyl\"{a}}
\email{enpasqua@jyu.fi}
\thanks{The author acknowledges the support by the Academy of Finland,
projects no.\ 314789 and no.\ 307333.}

\begin{document}
\date{\today} 
\begin{abstract}
The aim of this note is to study existence and main properties
of direct and inverse limits in the category of normed \(L^0\)-modules
(in the sense of Gigli) over a metric measure space.
\end{abstract}
\maketitle
\tableofcontents
\section*{Introduction}
Recent years have witnessed a growing interest of the mathematical
community towards the differential calculus on nonsmooth spaces. In this regard,
an important contribution is represented by N.\ Gigli's paper \cite{Gigli14},
wherein a first-order differential structure for metric measure
spaces has been proposed. Such theory is based upon the key notion of
\emph{normed \(L^0\)-module}, which provides a generalisation of the concept
of `space of measurable sections of a measurable Banach bundle'.
The main aim of the present manuscript is to prove that direct limits always exist in
the category of normed \(L^0\)-modules. Furthermore, we shall report the proof
of existence of inverse limits of normed \(L^0\)-modules, which has been
originally achieved in \cite{GPS18}. Finally, we will investigate the relation
between direct/inverse limits and other natural operations that are available
in this framework, such as dual and pullback.
\bigskip

\noindent\textit{Overview of the content.}
The concept of normed \(L^0\)-module that we are going to describe has been
originally introduced in \cite{Gigli14} and then further refined in \cite{Gigli17}.
We propose here an equivalent reformulation of its definition,
which is tailored to our purposes.

Let \((\X,\sfd,\mm)\) be a given metric measure space. Consider an algebraic
module \(\mathscr M\) over the commutative ring \(L^0(\mm)\) of all real-valued
Borel functions defined on \(\X\) (up to \(\mm\)-a.e.\ equality).
By \emph{pointwise norm} on \(\mathscr M\) we mean a map
\(|\cdot|\colon\mathscr M\to L^0(\mm)\) satisfying the following properties:
\[\begin{split}
|v|\geq 0\;\;\;\mm\text{-a.e.}&\quad\text{ for every }v\in\mathscr M,
\text{ with equality if and only if }v=0,\\
|v+w|\leq|v|+|w|\;\;\;\mm\text{-a.e.}&\quad\text{ for every }v,w\in\mathscr M,\\
|f\cdot v|=|f||v|\;\;\;\mm\text{-a.e.}&\quad\text{ for every }v\in\mathscr M
\text{ and }f\in L^0(\mm).
\end{split}\]
The pointwise norm \(|\cdot|\) can be naturally associated with a distance
\(\sfd_{\mathscr M}\) on \(\mathscr M\): chosen a Borel probability measure \(\mm'\)
on \(\X\) that is mutually absolutely continuous with respect to \(\mm\), we define
\[\sfd_{\mathscr M}(v,w)\coloneqq\int|v-w|\wedge 1\,\d\mm'
\quad\text{ for every }v,w\in\mathscr M.\]
Then we say that the couple \(\big(\mathscr M,|\cdot|\big)\) is
a \emph{normed \(L^0(\mm)\)-module} provided the relative metric space
\((\mathscr M,\sfd_{\mathscr M})\) is complete. The crucial example of normed
\(L^0\)-module one should keep in mind is the space of Borel vector fields
on a Riemannian manifold (with the usual pointwise operations).

Given two normed \(L^0(\mm)\)-modules \(\mathscr M\) and \(\mathscr N\),
we say that a map \(\varphi\colon\mathscr M\to\mathscr N\) is a \emph{morphism}
provided it is a morphism of \(L^0(\mm)\)-modules satisfying the inequality
\(\big|\varphi(v)\big|\leq|v|\) in the \(\mm\)-a.e.\ sense for every \(v\in\mathscr M\).
Consequently, we can consider the category of normed \(L^0(\mm)\)-modules.
The scope of these notes is to analyse direct and inverse limits in such category.
More in detail:
\begin{itemize}
\item[\(\rm i)\)] We prove that any direct system in the category of normed
\(L^0(\mm)\)-modules admits a direct limit (cf.\ Theorem \ref{thm:DL_nor_mod}).
Among other properties, we show (cf.\ Lemma \ref{lem:repr_mod}) that any normed
\(L^0(\mm)\)-module can be written as a direct limit of finitely-generated modules
(which is significant to the application b) we shall illustrate
at the end of this introduction) and (cf.\ Theorem \ref{thm:pullback_DL}) that
the direct limit functor commutes with the pullback functor.
\item[\(\rm ii)\)] Existence of inverse limits in the category of normed
\(L^0(\mm)\)-modules has been already proven by the author, together with
N.\ Gigli and E.\ Soultanis, in the paper \cite{GPS18}.
In order to make these notes self-contained, we shall recall the proof
of such fact in Theorem \ref{thm:IL_nor_mod}. We also examine several (not
previously known) properties of inverse limits in this setting; for instance,
we prove that `the dual of the direct limit coincides with the inverse limit
of the duals' (see Corollary \ref{cor:IDL_vs_dual}). On the other hand, inverse
limit functor and pullback functor do not commute (see Remark \ref{rmk:pullback_IL}).
\end{itemize}
It is worth to underline that the category of normed \(L^0(\mm)\)-modules
reduces to that of Banach spaces as soon as the reference measure \(\mm\)
is a Dirac measure \(\delta_{\bar x}\) concentrated on some point \(\bar x\in\X\),
whence the above-mentioned features of direct and inverse limits of normed
\(L^0\)-modules might be considered as a generalisation of the corresponding
ones for Banach spaces.
\bigskip

\noindent\textit{Motivation and related works.}
Besides the theoretical interest, the study of direct and inverse limits
in the category of normed \(L^0\)-modules is principally motivated by the following two
applications:
\begin{itemize}
\item[\(\rm a)\)] \textsc{Differential of a metric-valued locally Sobolev map.}
Any metric measure space \((\X,\sfd_\X,\mm)\) can be canonically associated
with a \emph{cotangent module} \(L^0_\mm(T^*\X)\) and a \emph{tangent module}
\(L^0_\mm(T\X)\), which are normed \(L^0(\mm)\)-modules that supply an
abstract notion of `measurable \(1\)-forms on \(\X\)' and `measurable
vector fields on \(\X\)', respectively; we refer the interested reader
to \cite{Gigli14,Gigli17} for a detailed description of such objects.
Moreover, there are several ways to define Sobolev maps from \((\X,\sfd_\X,\mm)\)
to a complete metric space \((\Y,\sfd_\Y)\). One of possible approaches is
via post-composition with Lipschitz functions (cf.\ \cite{HKST15}).
Given any map \(u\colon\X\to\Y\) that is locally Sobolev in the above sense,
one can always select a distinguished object \(|Du|\in L^2_{\rm loc}(\X,\sfd_\X,\mm)\)
-- called \emph{minimal weak upper gradient of \(u\)} -- which plays the role of
the `modulus of the differential of \(u\)'. The purpose of the work \cite{GPS18}
was to build the differential \(\d u\) associated to \(u\), defined as a
linear operator between (suitable variants of) tangent modules.
More precisely, in the special case in which \(|Du|\) is globally \(2\)-integrable
the differential of \(u\) is a map from \(L^0_\mm(T\X)\) to
\(\big(u^*L^0_\mu(T^*\Y)\big)^*\), where the measure \(\mu\) is defined as
\(\mu\coloneqq u_*\big(|Du|^2\,\mm\big)\). The precise choice of this finite
Borel measure \(\mu\) on \(\Y\) is due to the fact that it enjoys nice composition
properties. On the other hand, if the function \(|Du|\) is just locally \(2\)-integrable,
then the measure \(u_*\big(|Du|^2\,\mm\big)\) may no longer be \(\sigma\)-finite
(thus accordingly the cotangent module \(L^0_\mu(T^*\Y)\) is not well-defined).
The strategy one can adopt to overcome such difficulty is the following:
the family \(\mathcal F(u)\) of all open subsets \(\Omega\) of \(\X\)
satisfying \(\int_\Omega|Du|^2\,\d\mm<+\infty\) is partially ordered
by inclusion, whence the idea is to initially deal with the `approximating'
modules \(L^0_{\mu_\Omega}(T^*\Y)\) -- where we set
\(\mu_\Omega\coloneqq u_*\big(\nchi_\Omega\,|Du|^2\,\mm\big)\) --
and then pass to the inverse limit with respect to \(\Omega\in\mathcal F(u)\).
\item[\(\rm b)\)] \textsc{Concrete representation of a separable normed \(L^0\)-module.}
A significant way to build normed \(L^0\)-modules is to provide some
reasonable notion of measurable Banach bundle and consider the space of
its measurable sections (up to a.e.\ equality). Nevertheless, it is not clear
whether any normed \(L^0\)-module actually admits a similar representation.
In this direction, it is proven in \cite{LP18} that each finitely-generated
normed \(L^0\)-module can be viewed as the space of sections of some bundle.
The aim of the forthcoming paper \cite{DMLP19} is to extend this result to
all separable normed \(L^0\)-modules. One of the possible approaches to
achieve such goal is to realise any separable normed \(L^0\)-module
\(\mathscr M\) as the direct limit (with respect to a countable set of indices)
of finitely-generated normed \(L^0\)-modules \(\mathscr M_n\) and to apply the
previously known result to each module \(\mathscr M_n\).
\end{itemize}
\bigskip

\noindent\textit{Acknowledgements.}
I would like to thank Danka Lu\v{c}i\'{c} and Tapio Rajala for their
careful reading of a preliminary version of this manuscript.
\section{Preliminaries}
\subsection{Normed \texorpdfstring{\(L^0(\mm)\)}{L0(m)}-modules}
For our purposes, a \emph{metric measure space} is a triple
\((\X,\sfd,\mm)\), where \((\X,\sfd)\) is a complete and separable metric
space, while \(\mm\geq 0\) is a Radon measure on \((\X,\sfd)\).
We denote by \(L^0(\mm)\) the space of all Borel functions
\(f\colon\X\to\R\) considered up to \(\mm\)-a.e.\ equality.
It is well-known that \(L^0(\mm)\) is both a topological
vector space and a topological ring when equipped with the
usual pointwise operations and with the topology induced by the distance
\[\sfd_{L^0(\mm)}(f,g)\coloneqq\int|f-g|\wedge 1\,\d\mm'
\quad\text{ for every }f,g\in L^0(\mm),\]
where \(\mm'\) is any Borel probability measure on \(\X\) with \(\mm\ll\mm'\ll\mm\).
Given any (not necessarily countable) family \(\{f_i\}_{i\in I}\subseteq L^0(\mm)\),
we denote by \({\rm ess\,sup}_{i\in I}f_i\in L^0(\mm)\) and
\({\rm ess\,inf}_{i\in I}f_i\in L^0(\mm)\) its essential supremum and essential
infimum, respectively.
\begin{definition}[Pointwise norm]
Let \(\mathscr M\) be a module over the commutative ring \(L^0(\mm)\).
Then we say that a map \(|\cdot|\colon\mathscr M\to L^0(\mm)\) is a
\emph{pointwise seminorm} on \(\mathscr M\) provided
\[\begin{split}
|v|\geq 0&\quad\text{ for every }v\in\mathscr M,\\
|v+w|\leq|v|+|w|&\quad\text{ for every }v,w\in\mathscr M,\\
|f\cdot v|=|f||v|&\quad\text{ for every }v\in\mathscr M\text{ and }f\in L^0(\mm),
\end{split}\]
where all inequalities are intended in the \(\mm\)-a.e.\ sense.
Moreover, we say that \(|\cdot|\) is a \emph{pointwise norm} on \(\mathscr M\)
if in addition it holds that \(|v|=0\) \(\mm\)-a.e.\ if and only if \(v=0\).
\end{definition}

Any pointwise seminorm can be naturally associated with the following pseudometric:
\[\sfd_{\mathscr M}(v,w)\coloneqq\int|v-w|\wedge 1\,\d\mm'
\quad\text{ for every }v,w\in\mathscr M,\]
where \(\mm'\) is any given Borel probability measure on \(\X\)
such that \(\mm\ll\mm'\ll\mm\). It holds that \(\sfd_{\mathscr M}\)
is a distance if and only if \(|\cdot|\) is a pointwise norm.
\medskip

With this said, we can give a definition of normed \(L^0(\mm)\)-module
that is fully equivalent to the one that has been proposed in \cite{Gigli14,Gigli17}:
\begin{definition}[Normed \(L^0(\mm)\)-module]
A \emph{normed \(L^0(\mm)\)-module} is a module \(\mathscr M\) over \(L^0(\mm)\)
endowed with a pointwise norm \(|\cdot|\) whose associated distance
\(\sfd_{\mathscr M}\) is complete.
\end{definition}

A \emph{morphism} \(\varphi\colon\mathscr M\to\mathscr N\) between two
normed \(L^0(\mm)\)-modules \(\mathscr M\) and \(\mathscr N\) is any
\(L^0(\mm)\)-module morphism -- i.e.\ satisfying \(\varphi(f\cdot v)=f\cdot\varphi(v)\)
for every \(f\in L^0(\mm)\) and \(v\in\mathscr M\) -- such that
\[\big|\varphi(v)\big|\leq|v|\quad\text{ holds }
\mm\text{-a.e.\ for every }v\in\mathscr M.\]
This allows us to speak about the category of normed \(L^0(\mm)\)-modules.
\begin{example}\label{ex:Banach_are_NM}{\rm
Let us suppose that \(\mm=\delta_{\bar x}\) for some point \(\bar x\in\X\).
Then the ring \(L^0(\delta_{\bar x})\) can be canonically identified with the
field \(\R\), thus accordingly the category of normed \(L^0(\delta_{\bar x})\)-modules
is (equivalent to) the category of Banach spaces.
\fr}\end{example}
\begin{definition}[Generators]
Let \(\mathscr M\) be a normed \(L^0(\mm)\)-module. Then a family
\(S\subseteq\mathscr M\) is said to \emph{generate} \(\mathscr M\)
provided the smallest \(L^0(\mm)\)-module containing \(S\)
is \(\sfd_{\mathscr M}\)-dense in \(\mathscr M\), i.e.
\[\Big\{\sum\nolimits_{i=1}^nf_i\cdot v_i\;\Big|\;
n\in\N,\,(f_i)_{i=1}^n\subseteq L^0(\mm),\,(v_i)_{i=1}^n\subseteq S\Big\}
\quad\text{ is }\sfd_{\mathscr M}\text{-dense in }\mathscr M.\]
\end{definition}
\begin{lemma}[Metric identification]\label{lem:metric_identification}
Let \(\mathscr M\) be an \(L^0(\mm)\)-module with a pointwise
seminorm \(|\cdot|\). Consider the following equivalence relation on \(\mathscr M\):
given any \(v,w\in\mathscr M\), we declare that \(v\sim w\) provided \(|v-w|=0\)
holds \(\mm\)-a.e.\ on \(\X\). Then the quotient \(\mathscr M/\sim\)
inherits an \(L^0(\mm)\)-module structure and the map
\(\big|[v]_\sim\big|\coloneqq|v|\) is a pointwise norm on \(\mathscr M/\sim\).
\end{lemma}
\begin{proof}
The set \(\mathscr N\coloneqq\big\{v\in\mathscr M\,:\,|v|=0\;\mm\text{-a.e.}\big\}\)
is clearly a submodule of \(\mathscr M\), thus the quotient space
\(\mathscr M/\sim\,=\mathscr M/\mathscr N\) has a canonical \(L^0(\mm)\)-module
structure. Given that
\[\big||v|-|w|\big|\leq|v-w|\quad
\text{ holds }\mm\text{-a.e.\ for every }v,w\in\mathscr M,\]
the map \(|\cdot|\colon\mathscr M/\sim\,\to L^0(\mm)\) defined by
\(\big|[v]_\sim\big|\coloneqq|v|\) is well-posed and satisfies all
the pointwise norm axioms. This gives the statement.
\end{proof}
\begin{lemma}[Metric completion]\label{lem:metric_completion}
Let \(\mathscr M\) be an \(L^0(\mm)\)-module with a pointwise norm \(|\cdot|\).
Then there exists a unique (up to unique isomorphism) couple
\((\mathscr M_0,\iota)\), where
\begin{itemize}
\item[\(\rm i)\)] \(\mathscr M_0\) is a normed \(L^0(\mm)\)-module,
\item[\(\rm ii)\)] \(\iota\colon\mathscr M\to\mathscr M_0\)
is an \(L^0(\mm)\)-linear map preserving the pointwise norm,
\end{itemize}
such that the range \(\iota(\mathscr M)\) is dense in \(\mathscr M_0\)
with respect to the distance \(\sfd_{\mathscr M_0}\).
\end{lemma}
\begin{proof}
Denote by \((\mathscr M_0,\iota)\) the completion of the metric
space \((\mathscr M,\sfd_{\mathscr M})\), which
is known to be unique up to unique isomorphism. The pointwise norm
\(|\cdot|\colon\iota(\mathscr M)\to L^0(\mm)\) can be easily proved
to be continuous from \(\big(\iota(\mathscr M),\sfd_{\iota(\mathscr M)}\big)\)
to \(\big(L^0(\mm),\sfd_{L^0(\mm)}\big)\), whence it can be uniquely extended
to a continuous map \(|\cdot|\colon\mathscr M_0\to L^0(\mm)\).
Arguing by approximation, we conclude that the extended map \(|\cdot|\) is
a pointwise norm on \(\mathscr M_0\) and that
\(\sfd_{\mathscr M_0}(v,w)=\int|v-w|\wedge 1\,\d\mm'\)
for every \(v,w\in\mathscr M_0\). This proves the validity
of the statement.
\end{proof}

Given any two normed \(L^0(\mm)\)-modules \(\mathscr M\) and \(\mathscr N\),
we define the space \(\textsc{Hom}(\mathscr M,\mathscr N)\) as
\[\textsc{Hom}(\mathscr M,\mathscr N)\coloneqq
\Big\{T\colon\mathscr M\to\mathscr N\;\Big|\;T\text{ is }
L^0(\mm)\text{-linear and continuous}\Big\}.\]
Standard arguments show that for any \(T\in\textsc{Hom}(\mathscr M,\mathscr N)\)
there exists \(\ell\in L^0(\mm)\) such that
\begin{equation}\label{eq:pointwise_norm_dual}
\big|T(v)\big|\leq\ell\,|v|\quad\text{ holds }
\mm\text{-a.e.\ for every }v\in\mathscr M.
\end{equation}
It turns out that the function
\begin{equation}\label{eq:ptwse_norm_T}
|T|\coloneqq{\rm ess\,sup\,}\Big\{\big|T(v)\big|\;\Big|
\;v\in\mathscr M,\,|v|\leq 1\text{ holds }\mm\text{-a.e.}\Big\}\in L^0(\mm)
\end{equation}
is the minimal function \(\ell\) (in the \(\mm\)-a.e.\ sense) for which
\eqref{eq:pointwise_norm_dual} is satisfied.
We point out that an element \(T\in\textsc{Hom}(\mathscr M,\mathscr N)\)
is a morphism between \(\mathscr M\) and \(\mathscr N\) (in the categorical
sense) if and only if \(|T|\leq 1\) holds \(\mm\)-a.e.\ on \(\X\).
Furthermore, the space \(\textsc{Hom}(\mathscr M,\mathscr N)\) inherits a
natural structure of normed \(L^0(\mm)\)-module if endowed with the
pointwise operations
\[\begin{split}
(T+S)(v)\coloneqq T(v)+S(v)&
\quad\text{ for every }T,S\in\textsc{Hom}(\mathscr M,\mathscr N),\\
(f\cdot T)(v)\coloneqq f\cdot T(v)&\quad\text{ for every }f\in L^0(\mm)
\text{ and }T\in\textsc{Hom}(\mathscr M,\mathscr N)
\end{split}\]
and with the pointwise norm operator
\(\textsc{Hom}(\mathscr M,\mathscr N)\ni T\mapsto|T|\in L^0(\mm)\)
introduced in \eqref{eq:ptwse_norm_T}.
\begin{definition}[Dual of a normed \(L^0(\mm)\)-module]
Let \(\mathscr M\) be a normed \(L^0(\mm)\)-module.
Then we define its \emph{dual} normed \(L^0(\mm)\)-module \(\mathscr M^*\) as
\[\mathscr M^*\coloneqq\textsc{Hom}\big(\mathscr M,L^0(\mm)\big).\]
(Observe that \(L^0(\mm)\) itself can be viewed as a normed \(L^0(\mm)\)-module.)
\end{definition}

Let \(\mathscr M\), \(\mathscr N\) be any two normed \(L^0(\mm)\)-modules and
let \(\varphi\colon\mathscr M\to\mathscr N\) be a given morphism.
Then the \emph{adjoint} operator \(\varphi^{\rm adj}\colon\mathscr N^*\to\mathscr M^*\)
is defined as
\begin{equation}\label{eq:adjoint_map}
\varphi^{\rm adj}(\omega)\coloneqq\omega\circ\varphi
\quad\text{ for every }\omega\in\mathscr N^*.
\end{equation}
It is immediate to check that \(\varphi^{\rm adj}\)
is a morphism of normed \(L^0(\mm)\)-modules as well.
\begin{remark}\label{rmk:L0_B_valued}{\rm
Define \(\mathcal L_1\coloneqq\mathcal L^1\restr{[0,1]}\) and
consider a Banach space \(\B\). Then the space \(L^0\big([0,1],\B\big)\) of all
(strongly) Borel maps from \([0,1]\) to \(\B\) (considered up to
\(\mathcal L_1\)-a.e.\ equality) can be easily shown to be a normed
\(L^0(\mathcal L_1)\)-module if endowed with the following operations:
\[\begin{split}
(u+v)(t)&\coloneqq u(t)+v(t),\\
(f\cdot u)(t)&\coloneqq f(t)\,u(t),\\
|u|(t)&\coloneqq{\big\|u(t)\big\|}_\B
\end{split}\quad\text{ for }\mathcal L_1\text{-a.e.\ }t\in[0,1],\]
for every \(u,v\in L^0\big([0,1],\B\big)\) and \(f\in L^0(\mathcal L_1)\).
By combining the results of \cite[Section 1.6]{Gigli14} with the properties
of the \(L^0\)-completion studied in \cite{Gigli17}, one can deduce that
\(L^0\big([0,1],\B'\big)\) is isometrically embedded into
\(L^0\big([0,1],\B\big)^*\) and that
\begin{equation}\label{eq:RNP}
L^0\big([0,1],\B\big)^*\cong L^0\big([0,1],\B'\big)
\quad\Longleftrightarrow\quad
\B'\text{ has the Radon-Nikod\'{y}m property,}
\end{equation}
where \(\B'\) stands for the dual of \(\B\) as a Banach space.
(We refer to \cite{DiestelUhl77} for the definition of the Radon-Nikod\'{y}m property
and its main properties.)
\fr}\end{remark}
\begin{theorem}[Pullback of a normed \(L^0(\mm)\)-module]\label{thm:pullback}
Let \((\X,\sfd_\X,\mm_\X)\), \((\Y,\sfd_\Y,\mm_\Y)\) be metric measure spaces.
Let \(f\colon\X\to\Y\) be a Borel map with \(f_*\mm_\X\ll\mm_\Y\).
Then it holds that:
\begin{itemize}
\item[\(\rm i)\)] Let \(\mathscr M\) be a given  normed \(L^0(\mm_\Y)\)-module.
Then there exists a unique couple \((f^*\mathscr M,f^*)\)
-- where \(f^*\mathscr M\) is a normed \(L^0(\mm_\X)\)-module and
\(f^*\colon\mathscr M\to f^*\mathscr M\) is a linear map -- such that
\begin{equation}\label{eq:properties_pullback}\begin{split}
|f^*v|=|v|\circ f\;\;\;\mm\text{-a.e.}&\quad\text{ for every }v\in\mathscr M,\\
\{f^*v\,:\,v\in\mathscr M\}&\quad\text{ generates }f^*\mathscr M.
\end{split}\end{equation}
Uniqueness is up to unique isomorphism: given any other couple
\((\mathscr M_0,T)\) with the same properties, there is a unique normed
\(L^0(\mm_\X)\)-module morphism \(\Phi\colon f^*\mathscr M\to\mathscr M_0\) such that
\[\begin{tikzcd}
\mathscr M \arrow[r,"f^*"] \arrow[rd,swap,"T"] &
f^*\mathscr M \arrow[d,"\Phi"] \\
& \mathscr M_0
\end{tikzcd}\]
is a commutative diagram.
\item[\(\rm ii)\)] Let \(\mathscr M\), \(\mathscr N\) be two given normed
\(L^0(\mm_\Y)\)-modules. Let \(\varphi\colon\mathscr M\to\mathscr N\) be
a morphism of normed \(L^0(\mm_\Y)\)-modules. Then there exists a unique
morphism \(f^*\varphi\colon f^*\mathscr M\to f^*\mathscr N\) of normed
\(L^0(\mm_\X)\)-modules such that
\[\begin{tikzcd}
\mathscr M \arrow[r,"\varphi"] \arrow[d,swap,"f^*"] &
\mathscr N \arrow[d,"f^*"] \\
f^*\mathscr M \arrow[r,swap,"f^*\varphi"] & f^*\mathscr N
\end{tikzcd}\]
is a commutative diagram.
\end{itemize}
\end{theorem}
\begin{remark}{\rm
The notion of pullback of a normed \(L^0(\mm)\)-module introduced
in Theorem \ref{thm:pullback} above fits in the framework of category theory;
we refer to \cite[Remark 1.6.4]{Gigli14} for the details.
\fr}\end{remark}
\begin{example}\label{ex:special_case_pullback}{\rm
Consider two metric measure spaces \((\Y,\sfd_\Y,\mm_\Y)\),
\(({\rm Z},\sfd_{\rm Z},\mm_{\rm Z})\) with \(\mm_{\rm Z}\) finite.
We endow the space \(\X\coloneqq{\rm Z}\times\Y\) with the product distance
\(\sfd_\X=\sfd_{\rm Z}\times\sfd_\Y\), defined as
\[(\sfd_{\rm Z}\times\sfd_\Y)\big((z_1,y_1),(z_2,y_2)\big)
\coloneqq\sqrt{\sfd_{\rm Z}^2(z_1,z_2)+\sfd_\Y^2(y_1,y_2)}
\quad\text{ for every }(z_1,y_1),(z_2,y_2)\in\X,\]
and the product measure \(\mm_\X\coloneqq\mm_{\rm Z}\otimes\mm_\Y\).
Moreover, we call \(\pi\colon\X\to\Y\) the natural projection
map \((z,y)\mapsto y\), which is continuous and satisfies
\(\pi_*\mm_\X=\mm_{\rm Z}({\rm Z})\,\mm_\Y\ll\mm_\Y\).

Given any normed \(L^0(\mm_\Y)\)-module \(\mathscr M\), we define the space
\(L^0({\rm Z},\mathscr M)\) as the family of all (strongly) Borel maps
\(V\colon{\rm Z}\to\mathscr M\) considered up to \(\mm_{\rm Z}\)-a.e.\ equality.
It is straightforward to check that the space \(L^0({\rm Z},\mathscr M)\) is a
normed \(L^0(\mm_\X)\)-module if equipped with the following operations:
\[\begin{split}
(V+W)(z)\coloneqq V(z)+W(z)\in\mathscr M&
\quad\text{ for }\mm_{\rm Z}\text{-a.e.\ }z\in{\rm Z},\\
(f\cdot V)(z)\coloneqq f(z,\cdot)\cdot V(z)\in\mathscr M&
\quad\text{ for }\mm_{\rm Z}\text{-a.e.\ }z\in{\rm Z},\\
|V|(z,y)\coloneqq\big|V(z)\big|(y)&
\quad\text{ for }\mm_\X\text{-a.e.\ }(z,y)\in\X,
\end{split}\]
for every \(V,W\in L^0({\rm Z},\mathscr M)\) and \(f\in L^0(\mm_\X)\);
this constitutes a generalisation of what has been described
in Remark \ref{rmk:L0_B_valued}. Finally, we denote by
\({\rm T}\colon\mathscr M\to L^0({\rm Z},\mathscr M)\)
the linear operator sending any element \(v\in\mathscr M\) to the map
\({\rm T}(v)\colon{\rm Z}\to\mathscr M\) identically equal to \(v\).
We thus claim that
\begin{equation}\label{eq:pullback_special_case_claim}
\big(L^0({\rm Z},\mathscr M),{\rm T}\big)\cong(\pi^*\mathscr M,\pi^*).
\end{equation}
In order to prove it, we need to show that the two properties in
\eqref{eq:properties_pullback} are satisfied. For the first one,
notice that for any \(v\in\mathscr M\) it holds that
\[\big|{\rm T}(v)\big|(z,y)=\big|{\rm T}(v)(z)\big|(y)
=|v|(y)=|v|\big(\pi(z,y)\big)=\big(|v|\circ\pi\big)(z,y)
\quad\text{ for }\mm_\X\text{-a.e.\ }(z,y)\in\X.\]
For the second one, just observe that simple maps (i.e.\ Borel
maps from \(\rm Z\) to \(\mathscr M\) whose range is of finite cardinality)
are dense in \(L^0({\rm Z},\mathscr M)\). Therefore the claim
\eqref{eq:pullback_special_case_claim} is proven.
%The above argument is taken from what done in \cite[Section 1.6]{Gigli14}.
\fr}\end{example}
\subsection{Direct and inverse limits in a category}
The purpose of this subsection is to recall the notion
of direct/inverse limit in an arbitrary category; we
refer, for instance, to \cite{MacLane98} for a detailed account on this topic.
\bigskip

Fix a \emph{directed (partially ordered) set} \((I,\leq)\), which is
a nonempty partially ordered set such that any pair of elements admits an
upper bound (i.e.\ for every \(i,j\in I\) there exists \(k\in I\) satisfying
both \(i\leq k\) and \(j\leq k\)). The directed set \((I,\leq)\) can be considered
as a small category \(\mathcal I\), whose objects are the elements of \(I\)
and whose morphisms are defined as follows: given any \(i,j\in I\), there
is a (unique) morphism \(i\to j\) if and only if \(i\leq j\).
Let us also fix an arbitrary category \(\mathcal C\).
\bigskip

A \emph{direct system} in \(\mathcal C\) over \(I\) is any couple
\(\big(\{X_i\}_{i\in I},\{\varphi_{ij}\}_{i\leq j}\big)\),
where \(\{X_i\,:\,i\in I\}\) is a family of objects of \(\mathcal C\),
while \(\{\varphi_{ij}\,:\,i,j\in I,\,i\leq j\}\) is a family of
morphisms \(\varphi_{ij}\colon X_i\to X_j\) satisfying the following properties:
\begin{itemize}
\item[\(\rm i)\)] \(\varphi_{ii}\) is the identity of \(X_i\) for every \(i\in I\).
\item[\(\rm ii)\)] \(\varphi_{ik}=\varphi_{jk}\circ\varphi_{ij}\) for every
\(i,j,k\in I\) with \(i\leq j\leq k\).
\end{itemize}
Equivalently, a direct system in \(\mathcal C\) over \(I\) is a
covariant functor \(\mathcal I\to\mathcal C\).

We can define the \emph{direct limit} of the direct system
\(\big(\{X_i\}_{i\in I},\{\varphi_{ij}\}_{i\leq j}\big)\)
via a universal property. We say that
\(\big(\varinjlim X_\star,\{\varphi_i\}_{i\in I}\big)\) -- where \(\varinjlim X_\star\)
is an object of \(\mathcal C\) and \(\{\varphi_i\,:\,i\in I\}\) is a family
of morphisms \(\varphi_i\colon X_i\to\varinjlim X_\star\) called
\emph{canonical morphisms} -- is the direct limit
of \(\big(\{X_i\}_{i\in I},\{\varphi_{ij}\}_{i\leq j}\big)\) provided the
following properties hold:
\begin{itemize}
\item[\(\rm a)\)] \(\big(\varinjlim X_\star,\{\varphi_i\}_{i\in I}\big)\)
is a \emph{target}, i.e.\ the diagram
\[\begin{tikzcd}
X_i \arrow[r,"\varphi_{ij}"] \arrow[rd,swap,"\varphi_i"] &
X_j \arrow[d,"\varphi_j"] \\
& \varinjlim X_\star
\end{tikzcd}\]
commutes for every \(i,j\in I\) such that \(i\leq j\).
\item[\(\rm b)\)] Given any target \(\big(Y,\{\psi_i\}_{i\in I}\big)\),
there exists a unique morphism \(\Phi\colon\varinjlim X_\star\to Y\) such that
\[\begin{tikzcd}
X_i \arrow[r,"\varphi_i"] \arrow[rd,swap,"\psi_i"] &
\varinjlim X_\star \arrow[d,"\Phi"] \\
& Y
\end{tikzcd}\]
is a commutative diagram for every \(i\in I\).
\end{itemize}
In general, a direct system in an arbitrary category might not admit
a direct limit. Nevertheless, whenever the direct limit exists, it has
to be unique up to unique isomorphism: given any other direct limit
\(\big(X,\{\varphi'_i\}_{i\in I}\big)\) of
\(\big(\{X_i\}_{i\in I},\{\varphi_{ij}\}_{i\leq j}\big)\),
there is a unique isomorphism
\(\mathscr I\colon X\to\varinjlim X_\star\) such that
\(\varphi_i=\mathscr I\circ\varphi'_i\) holds for every \(i\in I\).
\bigskip

An \emph{inverse system} in \(\mathcal C\) over \(I\) is any couple
\(\big(\{X_i\}_{i\in I},\{\P_{ij}\}_{i\leq j}\big)\),
where \(\{X_i\,:\,i\in I\}\) is a family of objects of \(\mathcal C\),
while \(\{\P_{ij}\,:\,i,j\in I,\,i\leq j\}\) is a family of
morphisms \(\P_{ij}\colon X_j\to X_i\) satisfying the following properties:
\begin{itemize}
\item[\(\rm i)\)] \(\P_{ii}\) is the identity of \(X_i\) for every \(i\in I\).
\item[\(\rm ii)\)] \(\P_{ik}=\varphi_{ij}\circ\varphi_{jk}\)
for every \(i,j,k\in I\) with \(i\leq j\leq k\).
\end{itemize}
Equivalently, an inverse system in \(\mathcal C\) over \(I\) is a
contravariant functor \(\mathcal I\to\mathcal C\).

We can define the \emph{inverse limit} of the inverse system
\(\big(\{X_i\}_{i\in I},\{\P_{ij}\}_{i\leq j}\big)\)
via a universal property. We say that
\(\big(\varprojlim X_\star,\{\P_i\}_{i\in I}\big)\) -- where \(\varprojlim X_\star\)
is an object of \(\mathcal C\) and \(\{\P_i\,:\,i\in I\}\) is a family
of morphisms \(\P_i\colon\varprojlim X_\star\to X_i\) called
\emph{natural projections} -- is the inverse limit
of \(\big(\{X_i\}_{i\in I},\{\P_{ij}\}_{i\leq j}\big)\) provided the
following properties hold:
\begin{itemize}
\item[\(\rm a)\)] The diagram
\[\begin{tikzcd}
\varprojlim X_\star \arrow[rd,"\P_i"] \arrow[d,swap,"\P_j"] & \\
X_j \arrow[r,swap,"\P_{ij}"] & X_i
\end{tikzcd}\]
commutes for every \(i,j\in I\) such that \(i\leq j\).
\item[\(\rm b)\)] Given any other such couple \(\big(Y,\{{\rm Q}_i\}_{i\in I}\big)\)
-- namely satisfying \({\rm Q}_i=\P_{ij}\circ{\rm Q}_j\) for all \(i,j\in I\) with
\(i\leq j\) -- there exists a unique morphism
\(\Phi\colon Y\to\varprojlim X_\star\) such that
\[\begin{tikzcd}
Y \arrow[r,"\Phi"] \arrow[rd,swap,"{\rm Q}_i"] &
\varprojlim X_\star \arrow[d,"\P_i"] \\
& X_i
\end{tikzcd}\]
is a commutative diagram for every \(i\in I\).
\end{itemize}
In general, an inverse system in an arbitrary category does not necessarily admit
an inverse limit. Nevertheless, whenever the inverse limit exists, it has
to be unique up to unique isomorphism: given any other inverse limit
\(\big(X,\{\P'_i\}_{i\in I}\big)\), there exists
a unique isomorphism \(\mathscr I\colon X\to\varprojlim X_\star\)
such that \(\P'_i=\P_i\circ\mathscr I\) holds for every \(i\in I\).
\subsection{Direct and inverse limits of
\texorpdfstring{\(R\)}{R}-modules}\label{ss:lim_mod}
For the usefulness of the reader, we report here the construction of
the direct/inverse limit of (algebraic) modules over a commutative ring.
The material we are going to present can be found, e.g., in \cite{lang84}.
Let us fix a commutative ring \(R\) and a directed partially ordered set \((I,\leq)\).
\bigskip

Let \(\big(\{M_i\}_{i\in I},\{\varphi_{ij}\}_{i\leq j}\big)\)
be a direct system of \(R\)-modules over \(I\). We define an equivalence
relation \(\sim\) on \(\bigsqcup_{i\in I}M_i\): given \(v\in M_i\)
and \(w\in M_j\), we declare that \(v\sim w\) provided there is \(k\in I\) with
\(i,j\leq k\) such that \(\varphi_{ik}(v)=\varphi_{jk}(w)\). Then we define the
direct limit of \(\big(\{M_i\}_{i\in I},\{\varphi_{ij}\}_{i\leq j}\big)\) as
\[\varinjlim M_\star\coloneqq\bigsqcup_{i\in I}M_i\Big/\sim.\]
The \(R\)-module operations on \(\varinjlim M_\star\) are defined in the following way:
\begin{itemize}
\item Let \(\boldsymbol v,\boldsymbol w\in\varinjlim M_\star\) be fixed.
Pick any \(v\in\boldsymbol v\cap M_i\) and \(w\in\boldsymbol w\cap M_j\).
Choose some \(k\in I\) such that \(i,j\leq k\). Notice that
\(\varphi_{ik}(v)\in\boldsymbol v\cap M_k\) and
\(\varphi_{jk}(w)\in\boldsymbol w\cap M_k\). Then we define the sum
\(\boldsymbol v+\boldsymbol w\in\varinjlim M_\star\) as the equivalence class of
\(\varphi_{ik}(v)+\varphi_{jk}(w)\in M_k\).
\item Let \(\boldsymbol v\in\varinjlim M_\star\) and \(r\in R\) be fixed.
Pick any \(v\in\boldsymbol v\cap M_i\). Then we define
\(r\cdot\boldsymbol v\in\varinjlim M_\star\)
as the equivalence class of \(r\cdot v\in M_i\).
\end{itemize}
It is easy to check that such operations are well-posed
and the resulting structure \(\big(\varinjlim M_\star,+,\cdot\big)\) satisfies the
\(R\)-module axioms. The canonical morphisms \(\varphi_i\colon\,M_i\to\varinjlim M_\star\)
are obtained by sending each element to its equivalence class.
We point out a fundamental property:
\begin{equation}\label{eq:M_Alg_surj}
\text{For every }\boldsymbol v\in\varinjlim M_\star\text{ there exist }i\in I
\text{ and }v\in M_i\text{ such that }\varphi_i(v)=\boldsymbol v.
\end{equation}
The above claim is a direct consequence of the very definition of \(\varinjlim M_\star\).
\bigskip

Now let us consider an inverse system
\(\big(\{M_i\}_{i\in I},\{\P_{ij}\}_{i\leq j}\big)\) of \(R\)-modules over \(I\).
Then we define its inverse limit as
\[\varprojlim M_\star\coloneqq\Big\{v=\{v_i\}_{i\in I}\in\prod\nolimits_{i\in I}M_i\;\Big|
\;v_i=\P_{ij}(v_j)\text{ for every }i,j\in I\text{ with }i\leq j\Big\}.\]
The direct product \(\prod_{i\in I}M_i\) has a natural \(R\)-module structure
with respect to the element-wise operations. It can be readily shown that
\(\varprojlim M_\star\) is an \(R\)-submodule of \(\prod_{i\in I}M_i\). Finally,
the natural projections \(\P_i\colon\varprojlim M_\star\to M_i\) are defined as
\[\P_i(v)\coloneqq v_i\quad\text{ for every }
v=\{v_i\}_{i\in I}\in\varprojlim M_\star.\]
In particular, given any family \(\{v_i\}_{i\in I}\) such that \(v_i\in M_i\)
and \(v_i=\P_{ij}(v_j)\) hold for every \(i,j\in I\) with \(i\leq j\), there exists
a unique element \(v\in\varprojlim M_\star\) such that \(\P_i(v)=v_i\) for
every \(i\in I\).
\section{Direct limits of normed \texorpdfstring{\(L^0(\mm)\)}{L0(m)}-modules}
\subsection{Definition}
Unless otherwise specified, let \((\X,\sfd,\mm)\) be a fixed metric measure space.
The aim of this subsection is to prove that direct limits exist in the category
of normed \(L^0(\mm)\)-modules.
\begin{theorem}[Direct limit of normed \(L^0(\mm)\)-modules]\label{thm:DL_nor_mod}
Let \(\big(\{\mathscr M_i\}_{i\in I},\{\varphi_{ij}\}_{i\leq j}\big)\)
be a direct system of normed \(L^0(\mm)\)-modules. Then its direct limit
\(\big(\varinjlim\mathscr M_\star,\{\varphi_i\}_{i\in I}\big)\) exists
in the category of normed \(L^0(\mm)\)-modules.
\end{theorem}
\begin{proof}
Since \(\big(\{\mathscr M_i\}_{i\in I},\{\varphi_{ij}\}_{i\leq j}\big)\)
is a direct system in the category of algebraic \(L^0(\mm)\)-modules,
we can consider its direct limit
\(\big(\mathscr M_{\rm Alg},\{\varphi'_i\}_{i\in I}\big)\) in such category
(cf.\ Subsection \ref{ss:lim_mod}). It can be readily checked that
the following formula defines a pointwise seminorm on \(\mathscr M_{\rm Alg}\):
\begin{equation}\label{eq:ptwse_norm_on_M_Alg}
|\boldsymbol v|\coloneqq{\rm ess\,\inf\,}\big\{|v|\;:
\;i\in I,\,v\in\mathscr M_i,\,\varphi'_i(v)=\boldsymbol v\big\}
\quad\text{ for every }\boldsymbol v\in\mathscr M_{\rm Alg}.
\end{equation}
Clearly \(|\boldsymbol v|\in L^0(\mm)\) for every \(\boldsymbol v\in\mathscr M_{\rm Alg}\)
by \eqref{eq:M_Alg_surj}. Consider the equivalence relation \(\sim\)
on \(\mathscr M_{\rm Alg}\) as in Lemma \ref{lem:metric_identification} and
the metric completion \(\big(\varinjlim\mathscr M_\star,\iota\big)\) of
\(\mathscr M_{\rm Alg}/\sim\) as in Lemma \ref{lem:metric_completion}.
For \(i\in I\) we set the map
\(\varphi_i\colon\mathscr M_i\to\varinjlim\mathscr M_\star\) as
\(\varphi_i(v)\coloneqq\iota\big[\varphi'_i(v)\big]_\sim\in\varinjlim\mathscr M_\star\)
for all \(v\in\mathscr M_i\). We claim that
\[\big(\varinjlim\mathscr M_\star,\{\varphi_i\}_{i\in I}\big)\text{ is the direct limit of }
\big(\{\mathscr M_i\}_{i\in I},\{\varphi_{ij}\}_{i\leq j}\big)
\text{ as normed }L^0(\mm)\text{-modules.}\]
First of all, we know that \(\varinjlim\mathscr M_\star\) is a normed \(L^0(\mm)\)-module
from Lemmata \ref{lem:metric_identification} and \ref{lem:metric_completion}.
Each \(\varphi_i\) is \(L^0(\mm)\)-linear as composition of \(L^0(\mm)\)-module
morphisms, while for every \(v\in\mathscr M_i\) it holds that
\[\big|\varphi_i(v)\big|=\big|\varphi'_i(v)\big|
\overset{\eqref{eq:ptwse_norm_on_M_Alg}}\leq|v|
\quad\text{ in the }\mm\text{-a.e.\ sense,}\]
thus proving that \(\varphi_i\) is a normed \(L^0(\mm)\)-module morphism.
Given that \(\varphi'_j\circ\varphi_{ij}=\varphi'_i\) for all \(i,j\in I\)
with \(i\leq j\), we immediately deduce that \(\varphi_j\circ\varphi_{ij}=\varphi_i\)
as well. Therefore it only remains to prove the universal property:
let \(\big(\mathscr N,\{\psi_i\}_{i\in I}\big)\) be any given target.
It is a target even in the category of algebraic \(L^0(\mm)\)-modules,
therefore there exists a unique \(L^0(\mm)\)-morphism
\(\Phi'\colon\mathscr M_{\rm Alg}\to\mathscr N\) such that
\(\Phi'\circ\varphi'_i=\psi_i\) for every \(i\in I\). Then we are forced
to define the map \(\Phi\colon\iota(\mathscr M_{\rm Alg}/\sim)\to\mathscr N\) as
\begin{equation}\label{eq:DL_nor_mod_def_Phi}
\Phi\big(\iota[\boldsymbol v]_\sim\big)\coloneqq\Phi'(\boldsymbol v)
\quad\text{ for every }\boldsymbol v\in\mathscr M_{\rm Alg}.
\end{equation}
Observe that for any \(\boldsymbol v\in\mathscr M_{\rm Alg}\) we have
\[\big|\Phi'(\boldsymbol v)\big|=\big|\psi_i(v)\big|\leq|v|
\;\;\;\mm\text{-a.e.}\quad\text{ for every }i\in I\text{ and }
v\in\mathscr M_i\text{ with }\varphi'_i(v)=\boldsymbol v,\]
thus accordingly \(\big|\Phi'(\boldsymbol v)\big|\leq|\boldsymbol v|\)
holds \(\mm\)-a.e.\ in \(\X\). This grants that the operator \(\Phi\) in
\eqref{eq:DL_nor_mod_def_Phi} is well-defined and can be uniquely
extended to a normed \(L^0(\mm)\)-module morphism
\(\Phi\colon\varinjlim\mathscr M_\star\to\mathscr N\).
This proves the universal property and concludes the proof of the statement.
\end{proof}
\begin{remark}\label{rmk:DL_separable}{\rm
It follows from the proof of Theorem \ref{thm:DL_nor_mod} that
\(\bigcup_{i\in I}\varphi_i(\mathscr M_i)\) is
dense in \(\varinjlim\mathscr M_\star\).
In particular, if \(I\) is countable and each \(\mathscr M_i\) is separable,
then \(\varinjlim\mathscr M_\star\) is separable as well.
\fr}\end{remark}
\begin{definition}[Morphism of direct systems of normed \(L^0(\mm)\)-modules]
A morphism \(\Theta\) between two direct systems
\(\big(\{\mathscr M_i\}_{i\in I},\{\varphi_{ij}\}_{i\leq j}\big)\) and
\(\big(\{\mathscr N_i\}_{i\in I},\{\psi_{ij}\}_{i\leq j}\big)\) of normed
\(L^0(\mm)\)-modules is a family \(\Theta=\{\theta_i\}_{i\in I}\) of
normed \(L^0(\mm)\)-module morphisms \(\theta_i\colon\mathscr M_i\to\mathscr N_i\)
such that
\[\begin{tikzcd}
\mathscr M_i \arrow[r,"\theta_i"] \arrow[d,swap,"\varphi_{ij}"] &
\mathscr N_i \arrow[d,"\psi_{ij}"] \\
\mathscr M_j \arrow[r,swap,"\theta_j"] & \mathscr N_j
\end{tikzcd}\]
is a commutative diagram for every \(i,j\in I\) with \(i\leq j\).
\end{definition}

With the notion of morphism just introduced, it makes sense to consider
the category of direct systems of normed \(L^0(\mm)\)-modules. Then the
correspondence sending a direct system of normed \(L^0(\mm)\)-modules
to its direct limit can be made into a functor, as shown by the following result.
\begin{theorem}[The direct limit functor \(\varinjlim\)]
Let \(\Theta=\{\theta_i\}_{i\in I}\) be a morphism between two direct
systems \(\big(\{\mathscr M_i\}_{i\in I},\{\varphi_{ij}\}_{i\leq j}\big)\)
and \(\big(\{\mathscr N_i\}_{i\in I},\{\psi_{ij}\}_{i\leq j}\big)\) of
normed \(L^0(\mm)\)-modules, whose direct limits are denoted by
\(\big(\varinjlim\mathscr M_\star,\{\varphi_i\}_{i\in I}\big)\) and
\(\big(\varinjlim\mathscr N_\star,\{\psi_i\}_{i\in I}\big)\), respectively.
Then there exists a unique normed \(L^0(\mm)\)-module morphism
\(\varinjlim\theta_\star\colon\varinjlim\mathscr M_\star\to\varinjlim\mathscr N_\star\)
such that the diagram
\begin{equation}\label{eq:diagram_DL_morph}\begin{tikzcd}
\mathscr M_i \arrow[r,"\theta_i"] \arrow[d,swap,"\varphi_i"] &
\mathscr N_i \arrow[d,"\psi_i"] \\
\varinjlim\mathscr M_\star \arrow[r,swap,"\varinjlim\theta_\star"] &
\varinjlim\mathscr N_\star
\end{tikzcd}\end{equation}
commutes for every \(i\in I\). In particular, the
correspondence \(\varinjlim\) is a covariant functor
from the category of direct systems of normed \(L^0(\mm)\)-modules to
the category of normed \(L^0(\mm)\)-modules.
\end{theorem}
\begin{proof}
Let us denote by \(\big(\mathscr M_{\rm Alg},\{\varphi'_i\}_{i\in I}\big)\)
and \(\big(\mathscr N_{\rm Alg},\{\psi'_i\}_{i\in I}\big)\) the direct
limits (in the category of algebraic \(L^0(\mm)\)-modules) of
\(\big(\{\mathscr M_i\}_{i\in I},\{\varphi_{ij}\}_{i\leq j}\big)\) and
\(\big(\{\mathscr N_i\}_{i\in I},\{\psi_{ij}\}_{i\leq j}\big)\), respectively.
We define the map \(\theta'\colon\mathscr M_{\rm Alg}\to\mathscr N_{\rm Alg}\)
as follows: given any \(\boldsymbol v\in\mathscr M_{\rm Alg}\), there exist
\(i\in I\) and \(v\in\mathscr M_i\) such that \(\varphi'_i(v)=\boldsymbol v\)
by \eqref{eq:M_Alg_surj}, thus we set
\(\theta'(\boldsymbol v)\coloneqq(\psi'_i\circ\theta_i)(v)\).
It is straightforward to verify that \(\theta'\) is well-defined and is
the unique \(L^0(\mm)\)-module morphism such that
\(\theta'\circ\varphi'_i=\psi'_i\circ\theta_i\) for every \(i\in I\). As in the
proof of Theorem \ref{thm:DL_nor_mod}, let us consider the dense-range operators
\(\iota\colon\mathscr M_{\rm Alg}/\sim\,\to\varinjlim\mathscr M_\star\)
and \(\iota\colon\mathscr N_{\rm Alg}/\sim\,\to\varinjlim\mathscr N_\star\)
given by Lemmata \ref{lem:metric_identification} and \ref{lem:metric_completion}.
It can be readily checked that there exists a unique \(L^0(\mm)\)-module morphism
\(\theta\colon\iota(\mathscr M_{\rm Alg}/\sim)\to\iota(\mathscr N_{\rm Alg}/\sim)\)
such that
\[\theta\big(\iota[\boldsymbol v]_\sim\big)=
\iota\big[\theta'(\boldsymbol v)\big]_\sim
\quad\text{ for every }\boldsymbol v\in\mathscr M_{\rm Alg}.\]
Its well-posedness is granted by the \(\mm\)-a.e.\ inequality
\(\big|\iota\big[\theta'(\boldsymbol v)\big]_\sim\big|\leq
\big|\iota[\boldsymbol v]_\sim\big|\), which is satisfied for every
\(\boldsymbol v\in\mathscr M_{\rm Alg}\) as a consequence of the following observation:
\[\big|\iota\big[\theta'(\boldsymbol v)\big]_\sim\big|
=\big|\theta'(\boldsymbol v)\big|=\big|(\psi'_i\circ\theta_i)(v)\big|
\overset{\eqref{eq:ptwse_norm_on_M_Alg}}\leq\big|\theta_i(v)\big|\leq|v|
\quad\text{ holds }\mm\text{-a.e.}\]
for every \(i\in I\) and \(v\in\mathscr M_i\) such that \(\varphi'_i(v)=\boldsymbol v\),
whence \(\big|\iota\big[\theta'(\boldsymbol v)\big]_\sim\big|\leq|\boldsymbol v|
=\big|\iota[\boldsymbol v]_\sim\big|\) holds \(\mm\)-a.e.\ again by
\eqref{eq:ptwse_norm_on_M_Alg}. This also ensures that \(\theta\) can be uniquely
extended to a normed \(L^0(\mm)\)-module morphism
\(\varinjlim\theta_\star\colon\varinjlim\mathscr M_\star\to\varinjlim\mathscr N_\star\),
which is the unique morphism satisfying \(\varinjlim\theta_\star\circ\varphi_i=
\psi_i\circ\theta_i\) for every \(i\in I\). This concludes the proof of the statement.
\end{proof}
\subsection{Main properties}
In this subsection we collect the most important properties
of direct limits of normed \(L^0(\mm)\)-modules.
\begin{lemma}\label{lem:repr_mod}
Any normed \(L^0(\mm)\)-module can be written as direct limit of some
direct system of finitely-generated normed \(L^0(\mm)\)-modules.
\end{lemma}
\begin{proof}
Let \(\mathscr M\) be a normed \(L^0(\mm)\)-module.
Choose any set \(D\) that generates \(\mathscr M\) (possibly
\(D=\mathscr M\)). We denote by \(\mathcal P_F(D)\) the family
of all finite subsets of \(D\). Now choose any subset \(I\) of
\(\mathcal P_F(D)\) that is a directed partially ordered set
with respect to the inclusion relation \(\subseteq\) and such
that \(\bigcup_{F\in I}F\) generates \(\mathscr M\)
(for instance, \(\mathcal P_F(D)\) itself satisfies these properties).
Then let us define
\[\begin{split}
&\mathscr M_F\coloneqq\text{submodule of }\mathscr M\text{ generated by }F,\\
&\iota_{FG}\colon\mathscr M_F\hookrightarrow\mathscr M_G
\text{ inclusion map,}
\end{split}\]
for every \(F,G\in I\) with \(F\subseteq G\). It is then clear that
\(\big(\{\mathscr M_F\}_{F\in I},\{\iota_{FG}\}_{F\subseteq G}\big)\)
is a direct system of (finitely-generated) normed \(L^0(\mm)\)-modules.
We claim that
\begin{equation}\label{eq:repr_mod_claim}
\mathscr M\cong\varinjlim\mathscr M_\star,
\end{equation}
the canonical morphisms being given by the inclusion maps
\(\iota_F\colon\mathscr M_F\hookrightarrow\mathscr M\). First,
\(\big(\mathscr M,\{\iota_F\}_{F\in I}\big)\) is obviously
a target. To prove the universal property, fix another
target \(\big(\mathscr N,\{\psi_F\}_{F\in I}\big)\). Notice that the
\(L^0(\mm)\)-module \(\bigcup_{F\in I}\mathscr M_F\) is dense in
\(\mathscr M\) by construction. Therefore it can be readily checked
that there is a unique morphism \(\Phi\colon\mathscr M\to\mathscr N\)
such that \(\Phi(v)=\psi_F(v)\) holds for all \(F\in I\) and \(v\in\mathscr M_F\),
or equivalently \(\Phi\circ\iota_F=\psi_F\) for every \(F\in I\).
This proves the claim \eqref{eq:repr_mod_claim}.
\end{proof}
\begin{corollary}\label{cor:DL_sep_NM}
Let \(\mathscr M\) be a separable normed \(L^0(\mm)\)-module. Let \((v_n)_{n\in\N}\)
be a countable dense subset of \(\mathscr M\). Given any \(n,m\in\N\) with \(n\leq m\),
let us define:
\begin{itemize}
\item[\(\rm i)\)] \(\mathscr M_n\) as the module generated by \(\{v_1,\ldots,v_n\}\),
\item[\(\rm ii)\)] \(\iota_n\colon\mathscr M_n\hookrightarrow\mathscr M\) and
\(\iota_{nm}\colon\mathscr M_n\hookrightarrow\mathscr M_m\) as the inclusion maps.
\end{itemize}
Then it holds that \(\big(\mathscr M,\{\iota_n\}_{n\in\N}\big)\) is the direct limit
of the direct system \(\big(\{\mathscr M_n\}_{n\in\N},\{\iota_{nm}\}_{n\leq m}\big)\).
\end{corollary}
\begin{proof}
Notice that \(D\coloneqq(v_n)_{n\in\N}\) generates \(\mathscr M\).
Then the statement follows from the proof of Lemma \ref{lem:repr_mod} by
choosing as \(I\) the family of all subsets of \(D\) of the form
\(\{v_1,\ldots,v_n\}\) with \(n\in\N\).
%We define \(I\) as in the proof of Lemma \ref{lem:repr_mod} and \(I'\subseteq I\)
%as the set of all subsets of \(D\) of the form \(\{v_1,\ldots,v_n\}\) for some
%\(n\in\N\), so \((I',\subseteq)\) is a cofinal directed subset of \((I,\subseteq)\),
%i.e.\ for every \(F\in I\) there exists \(G\in I'\) with \(F\subseteq G\).
%This ensures that the direct limits of
%\(\big(\{\mathscr M_F\}_{F\in I},\{\iota_{FG}\}_{F\subseteq G}\big)\)
%and \(\big(\{\mathscr M_F\}_{F\in I'},\{\iota_{FG}\}_{F\subseteq G}\big)\) coincide.
%Now let us call \(\mathscr M_n\) the module generated by \(\{v_1,\ldots,v_n\}\)
%for every \(n\in\N\). Since \((\N,\leq)\cong(I',\subseteq)\) via the map
%\(n\mapsto\{v_1,\ldots,v_n\}\), we deduce from the proof of Lemma \ref{lem:repr_mod}
%that \(\mathscr M\cong\varinjlim\mathscr M_\star\).
\end{proof}

The category of normed \(L^0(\mm)\)-modules is a pointed category,
its zero object being the trivial space \(\{0\}\). Given two normed
\(L^0(\mm)\)-modules \(\mathscr M\), \(\mathscr N\) and a morphism
\(\varphi\colon\mathscr M\to\mathscr N\), it holds that:
\begin{itemize}
\item[\(\rm i)\)] The kernel of \(\varphi\) is the normed \(L^0(\mm)\)-submodule
\({\rm ker}(\varphi)\coloneqq\big\{v\in\mathscr M\,:\,\varphi(v)=0\big\}\)
of \(\mathscr M\) (together with the inclusion map
\({\rm ker}(\varphi)\hookrightarrow\mathscr M\)).
\item[\(\rm ii)\)] The image of \(\varphi\) is the normed \(L^0(\mm)\)-submodule
\({\rm im}(\varphi)\) of \(\mathscr N\) generated by the set-theoretic range
\(\varphi(\mathscr M)\) of \(\varphi\)
(together with the inclusion map \({\rm im}(\varphi)\hookrightarrow\mathscr N\)).
Observe that \(\varphi(\mathscr M)\) is an \(L^0(\mm)\)-submodule of
\(\mathscr N\), thus \({\rm im}(\varphi)\) coincides with the closure
of \(\varphi(\mathscr M)\) in \(\mathscr N\).
\end{itemize}
\begin{remark}{\rm
In general, the set-theoretic range of a normed \(L^0(\mm)\)-module
morphism might be not complete. For instance, consider the Banach spaces
\(\ell^\infty\) and \(c_0\), which can be regarded as normed
\(L^0(\mm)\)-modules provided the measure \(\mm\) is a Dirac delta
(as pointed out in Example \ref{ex:Banach_are_NM}). The linear
contraction \(\varphi\colon\ell^\infty\to c_0\), defined as
\[\varphi\big((t_k)_{k\in\N}\big)\coloneqq(t_k/k)_{k\in\N}
\quad\text{ for every }(t_k)_{k\in\N}\in\ell^\infty,\]
is injective and its range \(\varphi(\ell^\infty)\) is dense in \(c_0\).
The latter is granted by the following fact: the space \(c_{00}\)
(i.e.\ the space of all real-valued sequences having finitely many
non-zero terms) is dense in \(c_0\) and is contained in \(\varphi(\ell^\infty)\).
On the other hand, the operator \(\varphi\) cannot be surjective, as \(c_0\)
is separable while \(\ell^\infty\) is not. Therefore the normed space
\(\varphi(\ell^\infty)\) is not complete.
\fr}\end{remark}

The category of direct systems of normed \(L^0(\mm)\)-modules
is a pointed category, whose zero object is the direct system
\(\big(\{\mathscr M_i\}_{i\in I},\{\varphi_{ij}\}_{i\leq j}\big)\)
given by \(\mathscr M_i\coloneqq\{0\}\) for all \(i\in I\) and \(\varphi_{ij}\coloneqq 0\)
for all \(i,j\in I\) with \(i\leq j\). Given a morphism \(\Theta=\{\theta_i\}_{i\in I}\)
of two direct systems \(\big(\{\mathscr M_i\}_{i\in I},\{\varphi_{ij}\}_{i\leq j}\big)\)
and \(\big(\{\mathscr N_i\}_{i\in I},\{\psi_{ij}\}_{i\leq j}\big)\) of normed
\(L^0(\mm)\)-modules, it holds that:
\begin{itemize}
\item[\(\rm a)\)] The kernel \({\rm ker}(\Theta)\) of \(\Theta\) is given by
\(\big(\big\{{\rm ker}(\theta_i)\big\}_{i\in I},
\big\{\varphi_{ij}\restr{{\rm ker}(\theta_i)}\big\}_{i\leq j}\big)\).
\item[\(\rm b)\)] The image \({\rm im}(\Theta)\) of \(\Theta\) is given by
\(\big(\big\{{\rm im}(\theta_i)\big\}_{i\in I},
\big\{\psi_{ij}\restr{{\rm im}(\theta_i)}\big\}_{i\leq j}\big)\).
\end{itemize}
Items a) and b) make sense, since
\(\varphi_{ij}\big({\rm ker}(\theta_i)\big)\subseteq{\rm ker}(\theta_j)\)
and \(\psi_{ij}\big({\rm im}(\theta_i)\big)\subseteq{\rm im}(\theta_j)\)
whenever \(i\leq j\).
\begin{proposition}\label{prop:DL_preserves_surj}
Let \(\Theta=\{\theta_i\}_{i\in I}\) be a morphism between two direct systems
\(\big(\{\mathscr M_i\}_{i\in I},\{\varphi_{ij}\}_{i\leq j}\big)\) and
\(\big(\{\mathscr N_i\}_{i\in I},\{\psi_{ij}\}_{i\leq j}\big)\)
of normed \(L^0(\mm)\)-modules such that
\({\rm im}(\Theta)=\big(\{\mathscr N_i\}_{i\in I},\{\psi_{ij}\}_{i\leq j}\big)\).
Then it holds that \({\rm im}\big(\varinjlim\theta_\star\big)=\varinjlim\mathscr N_\star\).
\end{proposition}
\begin{proof}
First of all, we know that:
\begin{itemize}
\item \(\theta_i(\mathscr M_i)\) is dense in \(\mathscr N_i\) for every \(i\in I\),
as \({\rm im}(\theta_i)=\mathscr N_i\) by assumption.
\item \(\bigcup_{i\in I}\psi_i(\mathscr N_i)\) is dense in \(\varinjlim\mathscr N_\star\)
by Remark \ref{rmk:DL_separable}.
\end{itemize}
Hence \(\bigcup_{i\in I}(\theta\circ\varphi_i)(\mathscr M_i)\) is dense in
\(\varinjlim\mathscr N_\star\) by \eqref{eq:diagram_DL_morph}, where \(\theta\)
stands for \(\varinjlim\theta_\star\). This ensures that the set
\(\theta\big(\varinjlim\mathscr M_\star\big)\supseteq
\bigcup_{i\in I}(\theta\circ\varphi_i)(\mathscr M_i)\) is dense in
\(\varinjlim\mathscr N_\star\) as well, thus getting the statement.
\end{proof}
\begin{remark}{\rm
The dual statement of that of Proposition \ref{prop:DL_preserves_surj}
fails in general, since it is possible to build a morphism
\(\Theta=\{\theta_i\}_{i\in I}\) of direct systems with \({\rm ker}(\Theta)=0\)
such that \({\rm ker}\big(\varinjlim\theta_\star\big)\neq 0\).

For instance, suppose that \(\mm=\delta_{\bar x}\) for some \(\bar x\in\X\),
so that we are dealing with Banach spaces (as observed in Example
\ref{ex:Banach_are_NM}). Consider the sequence space \(\ell^2\) and
the morphism \(T\colon\ell^2\to\ell^2\) defined as
\(T(\lambda_1,\lambda_2,\lambda_3,\ldots)\coloneqq(0,\lambda_2,\lambda_3,\ldots)\).
Moreover, let us define the sequence \((a_n)_{n\in\N}\subseteq\ell^2\) as follows:
\(a_1\coloneqq(1/k)_{k\in\N}\) and \(a_n\coloneqq(\delta_{kn})_{k\in\N}\)
for all \(n\geq 2\). Then we set \(\mathscr M_n\coloneqq{\rm span}\{a_1,\ldots,a_n\}\)
and \(\mathscr N_n\coloneqq\ell^2\) for every \(n\in\N\), while for every \(n\leq m\)
we define the morphisms \(\varphi_{nm}\colon\mathscr M_n\to\mathscr M_m\) and
\(\psi_{nm}\colon\mathscr N_n\to\mathscr N_m\) as the inclusion map and the
identity map, respectively. Finally, let us define the morphism
\(\theta_n\colon\mathscr M_n\to\mathscr N_n\) as
\(\theta_n\coloneqq T\restr{\mathscr M_n}\) for every \(n\in\N\).
Therefore it is immediate to check that
\(\big(\{\mathscr M_n\}_{n\in\N},\{\varphi_{nm}\}_{n\leq m}\big)\),
\(\big(\{\mathscr N_n\}_{n\in\N},\{\psi_{nm}\}_{n\leq m}\big)\) are direct
systems of Banach spaces and that \(\Theta\coloneqq\{\theta_n\}_{n\in\N}\)
is a morphism between them satisfying \({\rm ker}(\Theta)=0\).
Obviously \(\varinjlim\mathscr N_\star=\ell^2\), but also
\(\varinjlim\mathscr M_\star=\ell^2\) by Corollary \ref{cor:DL_sep_NM} and by
density of the sequence \((a_n)_{n\in\N}\) in \(\ell^2\). It also turns out that
\(\varinjlim\theta_\star=T\). This yields the desired counterexample,
as the map \(T\) is not injective.
\fr}\end{remark}
\begin{lemma}\label{lem:DL_I_bdd}
Suppose that the directed set \((I,\leq)\) admits a greatest
element \(m\in I\). Then for any direct system
\(\big(\{\mathscr M_i\}_{i\in I},\{\varphi_{ij}\}_{i\leq j}\big)\)
of normed \(L^0(\mm)\)-modules it holds that
\begin{equation}\label{eq:DL_I_bdd}
\big(\mathscr M_m,\{\varphi_{im}\}_{i\in I}\big)
\text{ is the direct limit of }
\big(\{\mathscr M_i\}_{i\in I},\{\varphi_{ij}\}_{i\leq j}\big).
\end{equation}
In particular, given any morphism \(\Theta=\{\theta_i\}_{i\in I}\) between
two direct systems \(\big(\{\mathscr M_i\}_{i\in I},\{\varphi_{ij}\}_{i\leq j}\big)\)
and \(\big(\{\mathscr N_i\}_{i\in I},\{\psi_{ij}\}_{i\leq j}\big)\) of
normed \(L^0(\mm)\)-modules, it holds that \(\varinjlim\theta_\star=\theta_m\).
\end{lemma}
\begin{proof}
It easily follows from the fact that \(m\) is the greatest element of
\((I,\leq)\) that \(\big(\mathscr M_m,\{\varphi_{im}\}_{i\in I}\big)\)
is a target. To prove the universal property, fix another target
\(\big(\mathscr N,\{\psi_i\}_{i\in I}\big)\). Then \(\psi_m\) is the
unique normed \(L^0(\mm)\)-module morphism between \(\mathscr M_m\)
and \(\mathscr N\) such that \(\psi_m\circ\varphi_{im}=\psi_i\)
holds for every \(i\in I\), which shows the validity of the universal
property and accordingly the claim \eqref{eq:DL_I_bdd}.
\end{proof}
\begin{remark}\label{rmk:DL_no_full_faithful}{\rm
The direct limit functor \(\varinjlim\) is neither faithul nor full,
as we are going to prove.

Suppose that \(I=\{0,1\}\) and that \(\mm=\delta_{\bar x}\) for some \(\bar x\in\X\).
Set \(\mathscr M_0=\mathscr M_1=\mathscr N_0=\mathscr N_1\coloneqq\R^2\),
viewed as Banach spaces with the usual Euclidean
norm (recall Example \ref{ex:Banach_are_NM}). We also define
the maps \(\varphi_{01}\colon\mathscr M_0\to\mathscr M_1\) and
\(\psi_{01}\colon\mathscr N_0\to\mathscr N_1\) as \(\varphi_{01}(x,y)\coloneqq(x,y)\)
and \(\psi_{01}(x,y)\coloneqq(x,0)\), respectively. We have that
\(\varinjlim\mathscr M_\star=\mathscr M_1\) and
\(\varinjlim\mathscr N_\star=\mathscr N_1\) by Lemma \ref{lem:DL_I_bdd}.
\begin{itemize}
\item[\(\rm i)\)] Let us consider the morphisms \(\Theta=\{\theta_0,\theta_1\}\)
and \({\rm H}=\{\eta_0,\eta_1\}\) between the two direct systems
\(\big(\{\mathscr M_0,\mathscr M_1\},\{\varphi_{00},\varphi_{01},\varphi_{11}\}\big)\)
and \(\big(\{\mathscr N_0,\mathscr N_1\},\{\psi_{00},\psi_{01},\psi_{11}\}\big)\)
defined as follows:
\[\begin{split}
\theta_0(x,y)&\coloneqq(x,y),\\
\eta_0(x,y)&\coloneqq(x,-y),\\
\theta_1(x,y)=\eta_1(x,y)&\coloneqq(x,0).
\end{split}\]
Then \(\Theta\neq{\rm H}\), but \(\varinjlim\theta_\star=\varinjlim\eta_\star\)
by Lemma \ref{lem:DL_I_bdd}.
\item[\(\rm ii)\)] Consider the morphism \(\theta_1\colon\mathscr M_1\to\mathscr N_1\)
given by \(\theta_1(x,y)\coloneqq(x,y)\). Then there cannot exist a normed
\(L^0(\mm)\)-module morphism \(\theta_0\colon\mathscr M_0\to\mathscr N_0\)
such that \(\theta_1\circ\varphi_{01}=\psi_{01}\circ\theta_0\), the reason
being that the map \(\theta_1\circ\varphi_{01}\) is surjective while
\(\psi_{01}\) is not. This means that we cannot write the morphism
\(\theta_1\colon\varinjlim\mathscr M_i\to\varinjlim\mathscr N_i\)
as \(\varinjlim\theta_\star\) for some morphism \(\{\theta_0,\theta_1\}\)
between the direct systems
\(\big(\{\mathscr M_0,\mathscr M_1\},\{\varphi_{00},\varphi_{01},\varphi_{11}\}\big)\)
and \(\big(\{\mathscr N_0,\mathscr N_1\},\{\psi_{00},\psi_{01},\psi_{11}\}\big)\).
\end{itemize}
Items i) and ii) above show that the functor \(\varinjlim\) is neither faithul
nor full, respectively.
\fr}\end{remark}
\begin{theorem}[Pullback and direct limit commute]\label{thm:pullback_DL}
Let \((\X,\sfd_\X,\mm_\X)\), \((\Y,\sfd_\Y,\mm_\Y)\) be metric measure spaces.
Let \(f\colon\X\to\Y\) be a Borel map with \(f_*\mm_\X\ll\mm_\Y\).
Let \(\big(\{\mathscr M_i\}_{i\in I},\{\varphi_{ij}\}_{i\leq j}\big)\) be a direct
system of normed \(L^0(\mm_\Y)\)-modules, whose direct limit is denoted by
\(\big(\varinjlim\mathscr M_\star,\{\varphi_i\}_{i\in I}\big)\). Then
\(\big(\{f^*\mathscr M_i\}_{i\in I},\{f^*\varphi_{ij}\}_{i\leq j}\big)\) is
a direct system of normed \(L^0(\mm_\X)\)-modules. Its direct limit is
\begin{equation}\label{eq:DL_vs_pullback}
\varinjlim f^*\mathscr M_\star\cong f^*\varinjlim\mathscr M_\star
\end{equation}
together with the canonical morphisms \(\{f^*\varphi_i\}_{i\in I}\).
\end{theorem}
\begin{proof}
It follows from Theorem \ref{thm:pullback} that the diagram
\[\begin{tikzcd}
\mathscr M_i \arrow[r,"\varphi_{ij}"] \arrow[d,swap,"f^*"] &
\mathscr M_j \arrow[r,"\varphi_{jk}"] \arrow[d,swap,"f^*"] &
\mathscr M_k \arrow[r,"\varphi_k"] \arrow[d,swap,"f^*"] &
\varinjlim\mathscr M_\star \arrow[d,swap,"f^*"] \\
f^*\mathscr M_i \arrow[r,swap,"f^*\varphi_{ij}"] & f^*\mathscr M_j
\arrow[r,swap,"f^*\varphi_{jk}"] & f^*\mathscr M_k
\arrow[r,swap,"f^*\varphi_k"] & f^*\varinjlim\mathscr M_\star
\end{tikzcd}\]
commutes for every \(i,j,k\in I\) with \(i\leq j\leq k\), whence accordingly
\(\big(\{f^*\mathscr M_i\}_{i\in I},\{f^*\varphi_{ij}\}_{i\leq j}\big)\)
is a direct system of normed \(L^0(\mm_\X)\)-modules having
\(\big(f^*\varinjlim\mathscr M_\star,\{f^*\varphi_i\}_{i\in I}\big)\) as a target.
Given any other target \(\big(\mathscr N,\{\psi_i\}_{i\in I}\big)\), there exists
a unique morphism \(\Phi\colon f^*\varinjlim\mathscr M_\star\to\mathscr N\) such that
\begin{equation}\label{eq:pullback_DL_claim}
\Phi\big(f^*\big(\varphi_i(v)\big)\big)=\psi_i(f^*v)
\quad\text{ for every }i\in I\text{ and }v\in\mathscr M_i,
\end{equation}
as we are going to show. Since \(\bigcup_{i\in I}\varphi_i(\mathscr M_i)\)
is a dense submodule of \(\varinjlim\mathscr M_\star\) (cf.\ Remark
\ref{rmk:DL_separable}), we know that
\(\bigcup_{i\in I}\big\{f^*\big(\varphi_i(v)\big)\,:\,i\in I,\,v\in\mathscr M_i\big\}\)
generates \(f^*\varinjlim\mathscr M_\star\) by Theorem \ref{thm:pullback}, whence
uniqueness follows. To prove (well-posedness and) existence, we need to show that if
\(\boldsymbol v\in\varinjlim\mathscr M_\star\) can be written as
\(\boldsymbol v=\varphi_i(v)=\varphi_j(v')\) for some
\(v\in\mathscr M_i\) and \(v'\in\mathscr M_j\), then \(\psi_i(f^*v)=\psi_j(f^*v')\)
and the inequality \(\big|\psi_i(f^*v)\big|\leq|f^*\boldsymbol v|\) holds in the
\(\mm_\X\)-a.e.\ sense. Indeed, given any \(k\in I\) with \(i,j\leq k\) such that
\(\varphi_{ik}(v)=\varphi_{jk}(v')\), we deduce from the fact that the diagram
\[\begin{tikzcd}
\mathscr M_i \arrow[d,swap,"f^*"] \arrow[r,"\varphi_{ik}"] &
\mathscr M_k \arrow[d,swap,"f^*"] &
\mathscr M_j \arrow[l,swap,"\varphi_{jk}"] \arrow[d,swap,"f^*"] \\
f^*\mathscr M_i \arrow[rd,swap,"\psi_i"] \arrow[r,"f^*\varphi_{ik}"] &
f^*\mathscr M_k \arrow[d,swap,"\psi_k"] &
f^*\mathscr M_j \arrow[l,swap,"f^*\varphi_{jk}"] \arrow[ld,"\psi_j"] \\
& \mathscr N &
\end{tikzcd}\]
commutes that \(\psi_i(f^*v)\) and \(\psi_j(f^*v')\) coincide,
thus it makes sense to define \(\Phi(f^*\boldsymbol v)\coloneqq\psi_i(f^*v)\).
Since \(\big|\Phi(f^*\boldsymbol v)\big|\leq|f^*v|=|v|\circ f\) holds
\(\mm\)-a.e.\ for every \(i\in I\) and \(v\in\mathscr M_i\) with
\(\varphi_i(v)=\boldsymbol v\), we infer that
\(\big|\Phi(f^*\boldsymbol v)\big|\leq|\boldsymbol v|\circ f=|f^*\boldsymbol v|\).
Therefore there exists a (unique) morphism
\(\Phi\colon f^*\varinjlim\mathscr M_\star\to\mathscr N\)
satisfying \eqref{eq:pullback_DL_claim}, thus also
\(\Phi\circ(f^*\varphi_i)=\psi_i\) for all \(i\in I\).
This proves the universal property and accordingly that
\(\big(f^*\varinjlim\mathscr M_\star,\{f^*\varphi_i\}_{i\in I}\big)\)
is the direct limit of
\(\big(\{f^*\mathscr M_i\}_{i\in I},\{f^*\varphi_{ij}\}_{i\leq j}\big)\).
\end{proof}
\section{Inverse limits of normed \texorpdfstring{\(L^0(\mm)\)}{L0(m)}-modules}
\subsection{Definition}
Let us fix a metric measure space \((\X,\sfd,\mm)\). As we are going to see in
this subsection, inverse limits exist in the category of normed \(L^0(\mm)\)-modules.
This has been already proved in \cite{GPS18}; for the sake of completeness,
we report here the full proof of such fact.
\begin{theorem}[Inverse limit of normed \(L^0(\mm)\)-modules]\label{thm:IL_nor_mod}
Let \(\big(\{\mathscr M_i\}_{i\in I},\{\P_{ij}\}_{i\leq j}\big)\)
be an inverse system of normed \(L^0(\mm)\)-modules. Then its inverse limit
\(\big(\varprojlim\mathscr M_\star,\{\P_i\}_{i\in I}\big)\) exists
in the category of normed \(L^0(\mm)\)-modules.
\end{theorem}
\begin{proof}
Since \(\big(\{\mathscr M_i\}_{i\in I},\{\P_{ij}\}_{i\leq j}\big)\) is an inverse
system in the category of algebraic \(L^0(\mm)\)-modules, we can consider its inverse
limit \(\big(\mathscr M_{\rm Alg},\{\P'_i\}_{i\in I}\big)\) in such
category (cf.\ Subsection \ref{ss:lim_mod}).
Given any element \(v\in\mathscr M_{\rm Alg}\), we define
(up to \(\mm\)-a.e.\ equality) the Borel function \(|v|\colon\X\to[0,+\infty]\) as
\begin{equation}\label{eq:IL_def_ptwse_norm}
|v|\coloneqq\underset{i\in I}{\rm ess\,sup\,}\big|\P'_i(v)\big|
\end{equation}
Then we define the \(L^0(\mm)\)-submodule \(\varprojlim\mathscr M_\star\)
of \(\mathscr M_{\rm Alg}\) as
\[\varprojlim\mathscr M_\star\coloneqq
\big\{v\in\mathscr M_{\rm Alg}\,:\,|v|\in L^0(\mm)\big\}
=\big\{v\in\mathscr M_{\rm Alg}\,:\,|v|<+\infty\;\mm\text{-a.e.}\big\},\]
while the natural projections \(\P_i\colon\varprojlim\mathscr M_\star\to\mathscr M_i\)
are given by \(\P_i\coloneqq\P'_i\restr{\varprojlim\mathscr M_\star}\). We claim that
\begin{equation}\label{eq:IL_nor_mod_claim}
\big(\varprojlim\mathscr M_\star,\{\P_i\}_{i\in I}\big)\text{ is the inverse limit of }
\big(\{\mathscr M_i\}_{i\in I},\{\P_{ij}\}_{i\leq j}\big)
\text{ as normed }L^0(\mm)\text{-modules.}
\end{equation}
First of all, we need to show that \(\varprojlim\mathscr M_\star\) is a
normed \(L^0(\mm)\)-module. The only non-trivial fact to check is its completeness:
fix a Cauchy sequence \((v^n)_{n\in\N}\) in \(\varprojlim\mathscr M_\star\).
Given that \(\big|\P_i(v^n)\big|\leq|v^n|\) holds \(\mm\)-a.e.\ for all
\(i\in I\) and \(n\in\N\) by \eqref{eq:IL_def_ptwse_norm}, we deduce that the
sequence \(\big(\P_i(v^n)\big)_{n\in\N}\) is Cauchy in the complete space \(\mathscr M_i\)
for every \(i\in I\), whence it admits a limit \(v_i\in\mathscr M_i\).
Since the maps \(\P_{ij}\) are continuous for all \(i,j\in I\) with \(i\leq j\),
we can pass to the limit as \(n\to\infty\) in \(\P_i(v^n)=\P_{ij}\big(\P_j(v^n)\big)\)
and obtain that \(v_i=\P_{ij}(v_j)\), which means that
\(v\coloneqq\{v_i\}_{i\in I}\in\mathscr M_{\rm Alg}\).
Moreover, it can be readily checked that the map \(|\cdot|\) is a pointwise norm
on \(\varprojlim\mathscr M_\star\), thus the inequality
\(\big||v^n|-|v^m|\big|\leq|v^n-v^m|\) holds \(\mm\)-a.e.\ for every \(n,m\in\N\)
and accordingly \(\big(|v^n|\big)_{n\in\N}\) is a Cauchy sequence in
the space \(L^0(\mm)\). Calling \(f\in L^0(\mm)\) its limit,
we infer from \eqref{eq:IL_def_ptwse_norm} that
\[|v_i|=\lim_{n\to\infty}\big|\P_i(v^n)\big|\leq\lim_{n\to\infty}|v^n|=f
\;\;\;\mm\text{-a.e.}\quad\text{ for every }i\in I.\]
This grants that \(|v|\leq f<+\infty\) holds \(\mm\)-a.e.\ in \(\X\),
therefore \(v\in\varprojlim\mathscr M_\star\). It also holds that
\[|v-v^n|\overset{\eqref{eq:IL_def_ptwse_norm}}=
\underset{i\in I}{\rm ess\,sup\,}\big|v_i-\P_i(v^n)\big|
=\underset{i\in I}{\rm ess\,\sup}\lim_{m\to\infty}\big|\P_i(v^m)-\P_i(v^n)\big|
\overset{\eqref{eq:IL_def_ptwse_norm}}\leq\lim_{m\to\infty}|v^m-v^n|\]
in the \(\mm\)-a.e.\ sense. Then by letting \(n\to\infty\) we conclude that
\(|v-v^n|\to 0\) in \(L^0(\mm)\), or equivalently that \(v_n\to v\) in
\(\varprojlim\mathscr M_\star\), which proves the
completeness of \(\varprojlim\mathscr M_\star\).

Furthermore, it is immediate from the construction that each map \(\P_i\) is a normed
\(L^0(\mm)\)-module morphism and that \(\P_i=\P_{ij}\circ\P_j\) holds whenever
\(i,j\in I\) satisfy \(i\leq j\), thus in order to get the claim
\eqref{eq:IL_nor_mod_claim} it just remains to prove the universal property.
To this aim, fix any couple \(\big(\mathscr N,\{{\rm Q}_i\}_{i\in I}\big)\)
such that \({\rm Q}_i(w)=(\P_{ij}\circ{\rm Q}_j)(w)\) holds for all \(i,j\in I\)
with \(i\leq j\) and \(w\in\mathscr N\). Then for any \(w\in\mathscr N\)
there exists a unique element \(\Phi(w)\in\mathscr M_{\rm Alg}\) satisfying
\(\P'_i\big(\Phi(w)\big)={\rm Q}_i(w)\) for every \(i\in I\).
Given that \(\big|{\rm Q}_i(w)\big|\leq|w|\) holds
\(\mm\)-a.e.\ for every \(i\in I\), we deduce that
\[\big|\Phi(w)\big|\overset{\eqref{eq:IL_def_ptwse_norm}}=
\underset{i\in I}{\rm ess\,sup\,}\big|\P'_i\big(\Phi(w)\big)\big|
=\underset{i\in I}{\rm ess\,sup\,}\big|{\rm Q}_i(w)\big|
\leq|w|<+\infty\quad\text{ in the }\mm\text{-a.e.\ sense,}\]
whence \(\Phi(w)\in\varprojlim\mathscr M_\star\).
Therefore \(\Phi\colon\mathscr N\to\varprojlim\mathscr M_\star\)
is the unique morphism such that \({\rm Q}_i=\P_i\circ\Phi\) for all \(i\in I\).
This proves the universal property and accordingly \eqref{eq:IL_nor_mod_claim},
thus concluding the proof of the statement.
\end{proof}
\begin{remark}\label{rmk:IL_separable}{\rm
The following fact stems from the proof of Theorem \ref{thm:IL_nor_mod}:
if \(\{v_i\}_{i\in I}\) is a family of elements \(v_i\in\mathscr M_i\) satisfying
\(v_i=\P_{ij}(v_j)\) for all \(i,j\in I\) with \(i\leq j\) and
\({\rm ess\,sup}_{i\in I}|v_i|<+\infty\) in the \(\mm\)-a.e.\ sense, then there
exists a unique element \(v\in\varprojlim\mathscr M_\star\) such that \(v_i=\P_i(v)\)
for every \(i\in I\). Moreover, it holds that \(|v|={\rm ess\,sup}_{i\in I}|v_i|\).
\fr}\end{remark}
\begin{definition}[Morphism of inverse systems of normed \(L^0(\mm)\)-modules]
A morphism \(\Theta\) between two inverse systems
\(\big(\{\mathscr M_i\}_{i\in I},\{\P_{ij}\}_{i\leq j}\big)\) and
\(\big(\{\mathscr N_i\}_{i\in I},\{{\rm Q}_{ij}\}_{i\leq j}\big)\) of normed
\(L^0(\mm)\)-modules is a family \(\Theta=\{\theta_i\}_{i\in I}\) of
normed \(L^0(\mm)\)-module morphisms \(\theta_i\colon\mathscr M_i\to\mathscr N_i\)
such that
\begin{equation}\label{eq:inverse_system_norm_mod}\begin{tikzcd}
\mathscr M_j \arrow[r,"\theta_j"] \arrow[d,swap,"\P_{ij}"] &
\mathscr N_j \arrow[d,"{\rm Q}_{ij}"] \\
\mathscr M_i \arrow[r,swap,"\theta_i"] & \mathscr N_i
\end{tikzcd}\end{equation}
is a commutative diagram for every \(i,j\in I\) with \(i\leq j\).
\end{definition}

With the above notion of morphism at our disposal, we can consider
the category of inverse systems of normed \(L^0(\mm)\)-modules. The
correspondence associating to any inverse system of normed \(L^0(\mm)\)-modules
its inverse limit can be made into a functor, as we are going to see.
\begin{theorem}[The inverse limit functor \(\varprojlim\)]
Let \(\Theta=\{\theta_i\}_{i\in I}\) be a morphism between two inverse
systems \(\big(\{\mathscr M_i\}_{i\in I},\{\P_{ij}\}_{i\leq j}\big)\)
and \(\big(\{\mathscr N_i\}_{i\in I},\{{\rm Q}_{ij}\}_{i\leq j}\big)\)
of normed \(L^0(\mm)\)-modules, whose inverse limits are denoted by
\(\big(\varprojlim\mathscr M_\star,\{\P_i\}_{i\in I}\big)\) and
\(\big(\varprojlim\mathscr N_\star,\{{\rm Q}_i\}_{i\in I}\big)\), respectively.
Then there exists a unique normed \(L^0(\mm)\)-module morphism
\(\varprojlim\theta_\star\colon\varprojlim\mathscr M_\star\to\varprojlim\mathscr N_\star\)
such that the diagram
\begin{equation}\label{eq:diagram_IL_morph}\begin{tikzcd}
\varprojlim\mathscr M_\star \arrow[r,"\varprojlim\theta_\star"]
\arrow[d,swap,"\P_i"] & \varprojlim\mathscr N_\star \arrow[d,"{\rm Q}_i"] \\
\mathscr M_i \arrow[r,swap,"\theta_i"] & \mathscr N_i
\end{tikzcd}\end{equation}
commutes for every \(i\in I\). In particular, the
correspondence \(\varprojlim\) is a covariant functor
from the category of inverse systems of normed \(L^0(\mm)\)-modules to
the category of normed \(L^0(\mm)\)-modules.
\end{theorem}
\begin{proof}
Pick any \(v\in\varprojlim\mathscr M_\star\) and define
\(w_i\coloneqq\theta_i\big(\P_i(v)\big)\in\mathscr N_i\)
for all \(i\in I\). By \eqref{eq:inverse_system_norm_mod} we see that
\[{\rm Q}_{ij}(w_j)=({\rm Q}_{ij}\circ\theta_j)\big(\P_j(v)\big)
=(\theta_i\circ\P_{ij})\big(\P_j(v)\big)=\theta_i\big(\P_i(v)\big)
=w_i\quad\text{ for every }i,j\in I\text{ with }i\leq j.\]
Then there is a unique element
\(\big(\varprojlim\theta_\star\big)(v)=w\in\varprojlim\mathscr N_\star\)
such that \({\rm Q}_i(w)=w_i\) for every \(i\in I\), as observed in Remark
\ref{rmk:IL_separable}. One can readily check that the resulting map
\(\varprojlim\theta_\star\colon\varprojlim\mathscr M_\star\to\varprojlim\mathscr N_\star\)
is a morphism of normed \(L^0(\mm)\)-modules. Finally, it clearly holds that
\(\varprojlim\theta_\star\) is the unique morphism for which the diagram
\eqref{eq:diagram_IL_morph} is commutative for all \(i\in I\). Hence
the statement is achieved.
\end{proof}
\subsection{Main properties}
In this subsection we describe some important properties of inverse
limits in the category of normed \(L^0(\mm)\)-modules.
\begin{lemma}\label{lem:IL_trivial_case}
Let \(\mathscr M\neq\{0\}\) be a given normed \(L^0(\mm)\)-module.
Call \(\mathscr M_n\coloneqq\mathscr M\) for every \(n\in\N\).
For any \(n,m\in\N\) with \(n\leq m\), we define the morphism
\(\P_{nm}\colon\mathscr M_m\to\mathscr M_n\) as
\[\P_{nm}(v)\coloneqq\frac{n}{m}\,v\quad\text{ for every }v\in\mathscr M_m.\]
Then \(\big(\{\mathscr M_n\}_{n\in\N},\{\P_{nm}\}_{n\leq m}\big)\)
is an inverse system of normed \(L^0(\mm)\)-modules, with inverse limit
\[\varprojlim\mathscr M_\star=\{0\}.\]
\end{lemma}
\begin{proof}
It immediately follows from its very definition that
\(\big(\{\mathscr M_n\}_{n\in\N},\{\P_{nm}\}_{n\leq m}\big)\)
is an inverse system of normed \(L^0(\mm)\)-modules.
Moreover, its inverse limit \(\big(\mathscr M_{\rm Alg},\{\P'_n\}_{n\in\N}\big)\)
in the category of algebraic \(L^0(\mm)\)-modules is given by
\[\begin{split}
\mathscr M_{\rm Alg}&=\Big\{(nv)_{n\in\N}\in\prod\nolimits_{n\in\N}\mathscr M_n
\;\Big|\;v\in\mathscr M\Big\},\\
\P'_m\big((nv)_{n\in\N}\big)&=mv\quad\text{ for every }m\in\N
\text{ and }(nv)_{n\in\N}\in\mathscr M_{\rm Alg}.
\end{split}\]
Therefore for any element \(v\in\mathscr M_{\rm Alg}\) it \(\mm\)-a.e.\ holds that
\[|v|=\underset{n\in\N}{\rm ess\,sup\,}|nv|=+\infty\cdot\nchi_{\{|v|>0\}},\]
whence \(\varprojlim\mathscr M_\star=\{0\}\). This proves the statement.
\end{proof}

The category of inverse systems of normed \(L^0(\mm)\)-modules
is a pointed category, whose zero object is the inverse system
\(\big(\{\mathscr M_i\}_{i\in I},\{\P_{ij}\}_{i\leq j}\big)\)
given by \(\mathscr M_i\coloneqq\{0\}\) for all \(i\in I\) and \(\P_{ij}\coloneqq 0\)
for all \(i,j\in I\) with \(i\leq j\). Given a morphism \(\Theta=\{\theta_i\}_{i\in I}\)
of two inverse systems \(\big(\{\mathscr M_i\}_{i\in I},\{\P_{ij}\}_{i\leq j}\big)\)
and \(\big(\{\mathscr N_i\}_{i\in I},\{{\rm Q}_{ij}\}_{i\leq j}\big)\) of normed
\(L^0(\mm)\)-modules, it holds that:
\begin{itemize}
\item[\(\rm a)\)] The kernel \({\rm ker}(\Theta)\) of \(\Theta\) is given by
\(\big(\big\{{\rm ker}(\theta_i)\big\}_{i\in I},
\big\{\P_{ij}\restr{{\rm ker}(\theta_j)}\big\}_{i\leq j}\big)\).
\item[\(\rm b)\)] The image \({\rm im}(\Theta)\) of \(\Theta\) is given by
\(\big(\big\{{\rm im}(\theta_i)\big\}_{i\in I},
\big\{{\rm Q}_{ij}\restr{{\rm im}(\theta_j)}\big\}_{i\leq j}\big)\).
\end{itemize}
Items a) and b) make sense, as
\(\P_{ij}\big({\rm ker}(\theta_j)\big)\subseteq{\rm ker}(\theta_i)\)
and \({\rm Q}_{ij}\big({\rm im}(\theta_j)\big)\subseteq{\rm im}(\theta_i)\)
whenever \(i\leq j\).
\begin{proposition}\label{prop:IL_preserves_inj}
Let \(\Theta=\{\theta_i\}_{i\in I}\) be a morphism between inverse systems
\(\big(\{\mathscr M_i\}_{i\in I},\{\P_{ij}\}_{i\leq j}\big)\) and
\(\big(\{\mathscr N_i\}_{i\in I},\{{\rm Q}_{ij}\}_{i\leq j}\big)\)
of normed \(L^0(\mm)\)-modules such that \({\rm ker}(\Theta)=0\).
Then \({\rm ker}\big(\varprojlim\theta_\star\big)=0\).
\end{proposition}
\begin{proof}
Pick \(v\in\varprojlim\mathscr M_\star\) with
\(\big(\varprojlim\theta_\star\big)(v)=0\). This implies
\(\theta_i\big(\P_i(v)\big)={\rm Q}_i(0)=0\) for every \(i\in I\) by 
\eqref{eq:diagram_IL_morph}, whence \(\P_i(v)=0\) as \({\rm ker}(\theta_i)=0\)
by assumption. Then \(v=0\) by Remark \ref{rmk:IL_separable}.
\end{proof}
\begin{remark}{\rm
The dual statement of that of Proposition \ref{prop:IL_preserves_inj} is false,
as one can build a morphism of inverse systems \(\Theta=\{\theta_i\}_{i\in I}\) with
\({\rm im}(\Theta)=\big(\{\mathscr N_i\}_{i\in I},\{{\rm Q}_{ij}\}_{i\leq j}\big)\) such that \({\rm im}\big(\varprojlim\theta_\star\big)\neq\varprojlim\mathscr N_\star\).

For instance, fix any normed \(L^0(\mm)\)-module \(\mathscr M\neq\{0\}\).
Let us define the inverse systems of normed \(L^0(\mm)\)-modules
\(\big(\{\mathscr M_n\}_{n\in\N},\{\P_{nm}\}_{n\leq m}\big)\) and
\(\big(\{\mathscr N_n\}_{n\in\N},\{{\rm Q}_{nm}\}_{n\leq m}\big)\) as follows:
\[\begin{split}
\mathscr M_n=\mathscr N_n\coloneqq\mathscr M&\quad\text{ for every }n\in\N,\\
\P_{nm}(v)\coloneqq\frac{n}{m}\,v&\quad\text{ for every }n\leq m
\text{ and }v\in\mathscr M_m,\\
{\rm Q}_{nm}(w)\coloneqq w&\quad\text{ for every }n\leq m
\text{ and }w\in\mathscr N_m.
\end{split}\]
The morphism \(\Theta=\{\theta_n\}_{n\in\N}\) between 
\(\big(\{\mathscr M_n\}_{n\in\N},\{\P_{nm}\}_{n\leq m}\big)\) and
\(\big(\{\mathscr N_n\}_{n\in\N},\{{\rm Q}_{nm}\}_{n\leq m}\big)\)
we consider is given by
\[\theta_n(v)\coloneqq\frac{1}{n}\,v
\quad\text{ for every }n\in\N\text{ and }v\in\mathscr M_n.\]
Therefore \(\varprojlim\mathscr M_\star=\{0\}\) by Lemma \ref{lem:IL_trivial_case}
and \(\varprojlim\mathscr N_\star=\mathscr M\). This yields the desired
counterexample, as all the maps \(\theta_n\) are surjective but
\({\rm im}\big(\varprojlim\theta_\star\big)=\{0\}\neq\mathscr M\).
\fr}\end{remark}
\begin{lemma}\label{lem:IL_I_bdd}
Suppose that the directed set \((I,\leq)\) admits a greatest
element \(m\in I\). Then for any inverse system
\(\big(\{\mathscr M_i\}_{i\in I},\{\P_{ij}\}_{i\leq j}\big)\)
of normed \(L^0(\mm)\)-modules it holds that
\begin{equation}\label{eq:IL_I_bdd}
\big(\mathscr M_m,\{\P_{im}\}_{i\in I}\big)
\text{ is the inverse limit of }
\big(\{\mathscr M_i\}_{i\in I},\{\P_{ij}\}_{i\leq j}\big).
\end{equation}
In particular, given any morphism \(\Theta=\{\theta_i\}_{i\in I}\) between
two inverse systems \(\big(\{\mathscr M_i\}_{i\in I},\{\P_{ij}\}_{i\leq j}\big)\)
and \(\big(\{\mathscr N_i\}_{i\in I},\{{\rm Q}_{ij}\}_{i\leq j}\big)\) of
normed \(L^0(\mm)\)-modules, it holds that \(\varprojlim\theta_\star=\theta_m\).
\end{lemma}
\begin{proof}
Fix any couple \(\big(\mathscr N,\{{\rm Q}_i\}_{i\in I}\big)\)
-- where \(\mathscr N\) is a normed \(L^0(\mm)\)-module and
\({\rm Q}_i\colon\mathscr N\to\mathscr M_i\) are morphisms --
satisfying \({\rm Q}_i=\P_{ij}\circ{\rm Q}_j\) for all \(i,j\in I\)
with \(i\leq j\). Hence \(\Phi\coloneqq{\rm Q}_m\) is the unique morphism from
\(\mathscr N\) to \(\mathscr M_m\) such that \({\rm Q}_i=\P_{im}\circ\Phi\)
holds for every \(i\in I\), which proves \eqref{eq:IL_I_bdd}.
\end{proof}
\begin{remark}\label{rmk:IL_no_full_faithful}{\rm
By slightly modifying the examples provided in Remark \ref{rmk:DL_no_full_faithful},
it can be readily checked that the inverse limit functor \(\varprojlim\) is neither
faithul nor full.
\fr}\end{remark}
\begin{proposition}\label{prop:IDL_vs_dual}
Let \(\big(\{\mathscr M_i\}_{i\in I},\{\varphi_{ij}\}_{i\leq j}\big)\)
be a direct system of normed \(L^0(\mm)\)-modules, whose direct limit
is denoted by \(\big(\varinjlim\mathscr M_\star,\{\varphi_i\}_{i\in I}\big)\).
Given any normed \(L^0(\mm)\)-module \(\mathscr N\) and \(i,j\in I\)
with \(i\leq j\), we define the morphism \(\P_{ij}\colon
\textsc{Hom}(\mathscr M_j,\mathscr N)\to\textsc{Hom}(\mathscr M_i,\mathscr N)\) as
\[\P_{ij}(T)\coloneqq T\circ\varphi_{ij}
\quad\text{ for every }T\in\textsc{Hom}(\mathscr M_j,\mathscr N).\]
Then \(\big(\big\{\textsc{Hom}(\mathscr M_i,\mathscr N)\big\}_{i\in I},
\{\P_{ij}\}_{i\leq j}\big)\) is an inverse system of normed
\(L^0(\mm)\)-modules. Moreover, it holds that
\[\varprojlim\textsc{Hom}(\mathscr M_\star,\mathscr N)\cong
\textsc{Hom}\big(\varinjlim\mathscr M_\star,\mathscr N\big),\]
the natural projections
\(\P_i\colon\textsc{Hom}\big(\varinjlim\mathscr M_\star,\mathscr N\big)
\to\textsc{Hom}(\mathscr M_i,\mathscr N)\) being defined as
\(\P_i(T)\coloneqq T\circ\varphi_i\) for every \(i\in I\) and
\(T\in\textsc{Hom}\big(\varinjlim\mathscr M_\star,\mathscr N\big)\).
\end{proposition}
\begin{proof}
Given any \(i,j,k\in I\) with \(i\leq j\leq k\) and
\(T\in\textsc{Hom}(\mathscr M_k,\mathscr N)\), it holds that
\[\P_{ik}(T)=T\circ\varphi_{ik}=T\circ\varphi_{jk}\circ\varphi_{ij}
=\P_{ij}(T\circ\varphi_{jk})=(\P_{ij}\circ\P_{jk})(T),\]
whence \(\big(\big\{\textsc{Hom}(\mathscr M_i,\mathscr N)\big\}_{i\in I},
\{\P_{ij}\}_{i\leq j}\big)\) is an inverse system. Analogously,
\(\P_i=\P_{ij}\circ\P_j\) holds for all \(i,j\in I\) with
\(i\leq j\), thus to conclude it only remains to show that
\(\big(\textsc{Hom}\big(\varinjlim\mathscr M_\star,\mathscr N\big),
\{\P_i\}_{i\in I}\big)\) satisfies the universal property defining the inverse limit.
Fix any \(\big(\mathscr O,\{{\rm Q}_i\}_{i\in I}\big)\),
where \(\mathscr O\) is a normed \(L^0(\mm)\)-module, while the morphisms
\({\rm Q}_i\colon\mathscr O\to\textsc{Hom}(\mathscr M_i,\mathscr N)\)
satisfy \({\rm Q}_i=\P_{ij}\circ{\rm Q}_j\) for every \(i,j\in I\) with
\(i\leq j\). Given any element \(w\in\mathscr O\), we consider the
family \(\big\{{\rm Q}_i(w)\big\}_{i\in I}\). Call \(f_w\) the function
\(\nchi_{\{|w|>0\}}\frac{1}{|w|}\in L^0(\mm)\) and observe that the
morphisms \(f_w\cdot{\rm Q}_i(w)\colon\mathscr M_i\to\mathscr N\) satisfy
\[\begin{split}
\big(f_w\cdot{\rm Q}_i(w)\big)(v)&=f_w\cdot{\rm Q}_i(w)(v)
=f_w\cdot\P_{ij}\big({\rm Q}_j(w)\big)(v)
=f_w\cdot\big({\rm Q}_j(w)\circ\varphi_{ij}\big)(v)\\
&=f_w\cdot{\rm Q}_j(w)\big(\varphi_{ij}(v)\big)
=\big(f_w\cdot{\rm Q}_j(w)\big)\big(\varphi_{ij}(v)\big)
\end{split}\]
for every \(i,j\in I\) with \(i\leq j\) and \(v\in\mathscr M_i\). This shows
that \(\big(\mathscr N,\big\{f_w\cdot{\rm Q}_i(w)\big\}_{i\in I}\big)\) is a target for
the direct system \(\big(\{\mathscr M_i\}_{i\in I},\{\varphi_{ij}\}_{i\leq j}\big)\),
whence there exists a unique morphism
\(\Phi_0(w)\colon\varinjlim\mathscr M_\star\to\mathscr N\) such that
\(f_w\cdot{\rm Q}_i(w)(v)=\Phi_0(w)\big(\varphi_i(v)\big)\) holds for every
\(i\in I\) and \(v\in\mathscr M_i\), thus accordingly the element
\(\Phi(w)\coloneqq|w|\cdot\Phi_0(w)\in
\textsc{Hom}\big(\varinjlim\mathscr M_\star,\mathscr N\big)\) satisfies
\({\rm Q}_i(w)(v)=\Phi(w)\big(\varphi_i(v)\big)=\P_i\big(\Phi(w)\big)(v)\)
for all \(i\in I\) and \(v\in\mathscr M_i\). Hence
\(\Phi\colon\mathscr O\to\textsc{Hom}\big(\varinjlim\mathscr M_\star,\mathscr N\big)\)
is the unique morphism such that
\[\begin{tikzcd}
\mathscr O \arrow[r,"\Phi"] \arrow[rd,swap,"{\rm Q}_i"] &
\textsc{Hom}\big(\varinjlim\mathscr M_\star,\mathscr N\big) \arrow[d,"\P_i"] \\
& \textsc{Hom}(\mathscr M_i,\mathscr N)
\end{tikzcd}\]
is a commutative diagram for any \(i\in I\). Then the statement is achieved.
\end{proof}
\begin{corollary}\label{cor:IDL_vs_dual}
Let \(\big(\{\mathscr M_i\}_{i\in I},\{\varphi_{ij}\}_{i\leq j}\big)\)
be a direct system of normed \(L^0(\mm)\)-modules, whose direct limit
is denoted by \(\big(\varinjlim\mathscr M_\star,\{\varphi_i\}_{i\in I}\big)\).
Then \(\big(\{\mathscr M^*_i\}_{i\in I},\{\varphi^{\rm adj}_{ij}\}_{i\leq j}\big)\)
is an inverse system of normed \(L^0(\mm)\)-modules (cf.\ \eqref{eq:adjoint_map}
for the definition of the adjoint \(\varphi^{\rm adj}_{ij}\)). Moreover, it holds that
\begin{equation}\label{eq:IDL_vs_dual}
\varprojlim\mathscr M^*_\star\cong\big(\varinjlim\mathscr M_\star\big)^*,
\end{equation}
the natural projections being given by
\(\varphi^{\rm adj}_i\colon\big(\varinjlim\mathscr M_\star\big)^*\to\mathscr M^*_i\)
for every \(i\in I\).
\end{corollary}
\begin{proof}
Just apply Proposition \ref{prop:IDL_vs_dual} with \(\mathscr N\coloneqq L^0(\mm)\).
\end{proof}
\begin{remark}[Pullback and inverse limit do not commute]\label{rmk:pullback_IL}{\rm
Let \((\X,\sfd_\X,\mm_\X)\), \((\Y,\sfd_\Y,\mm_\Y)\) be metric measure spaces.
Let \(f\colon\X\to\Y\) be a Borel map with \(f_*\mm_\X\ll\mm_\Y\).
Let \(\big(\{\mathscr M_i\}_{i\in I},\{\P_{ij}\}_{i\leq j}\big)\) be an inverse
system of normed \(L^0(\mm_\Y)\)-modules, whose inverse limit is denoted by
\(\big(\varprojlim\mathscr M_\star,\{\P_i\}_{i\in I}\big)\). Then
\(\big(\{f^*\mathscr M_i\}_{i\in I},\{f^*\P_{ij}\}_{i\leq j}\big)\) is
an inverse system of normed \(L^0(\mm_\X)\)-modules, as the diagram
\[\begin{tikzcd}
\varprojlim\mathscr M_\star \arrow[r,"\P_k"] \arrow[d,swap,"f^*"] &
\mathscr M_k \arrow[r,"\P_{jk}"] \arrow[d,swap,"f^*"] &
\mathscr M_j \arrow[r,"\P_{ij}"] \arrow[d,swap,"f^*"] &
\mathscr M_i \arrow[d,swap,"f^*"] \\
f^*\varprojlim\mathscr M_\star \arrow[r,swap,"f^*\P_k"] & f^*\mathscr M_k
\arrow[r,swap,"f^*\P_{jk}"] & f^*\mathscr M_j
\arrow[r,swap,"f^*\P_{ij}"] & f^*\mathscr M_i
\end{tikzcd}\]
commutes for every \(i,j,k\in I\) with \(i\leq j\leq k\) by
Theorem \ref{thm:pullback}. Nevertheless, it might happen that
\[\varprojlim f^*\mathscr M_\star\ncong f^*\varprojlim\mathscr M_\star.\]
For instance, consider the constant map \(\pi\colon[0,1]\to\{0\}\),
where the domain is endowed with the Euclidean distance and the
restricted Lebesgue measure \(\mathcal L_1=\mathcal L^1\restr{[0,1]}\),
while the target is endowed with the trivial distance and the Dirac delta
measure \(\delta_0\). Notice that trivially \(\pi_*\mathcal L_1\ll\delta_0\).
Moreover, by combining Example \ref{ex:special_case_pullback} with
Example \ref{ex:Banach_are_NM} we deduce that
\begin{equation}\label{eq:pullback_special_case}
\pi^*\B\cong L^0\big([0,1],\B\big)\quad\text{ for every Banach space }\B.
\end{equation}
Now consider the Banach space \(\ell^1\), that is the direct limit of some direct
system \(\big(\{\B_i\}_{i\in I},\{\varphi_{ij}\}_{i\leq j}\big)\)
of finite-dimensional Banach spaces, for instance by Lemma \ref{lem:repr_mod}.
Since the spaces \(\B'_i\) have the Radon-Nikod\'{y}m property while
\(\ell^\infty=(\ell^1)'\) does not (cf.\ \cite{DiestelUhl77}),
we know from \eqref{eq:RNP} that
\begin{equation}\label{eq:RNP_ex}\begin{split}
L^0\big([0,1],\B_i\big)^*&\cong L^0\big([0,1],\B'_i\big)\quad\text{ for every }i\in I,\\
L^0\big([0,1],\ell^1\big)^*&\ncong L^0\big([0,1],\ell^\infty\big).
\end{split}\end{equation}
Therefore it holds that
\[\begin{split}
\pi^*\varprojlim\B'_\star&\overset{\eqref{eq:IDL_vs_dual}}\cong
\pi^*\big(\varinjlim\B_\star\big)'\cong\pi^*\ell^\infty
\overset{\eqref{eq:pullback_special_case}}\cong L^0\big([0,1],\ell^\infty\big)
\overset{\eqref{eq:RNP_ex}}\ncong L^0\big([0,1],\ell^1\big)^*
\overset{\eqref{eq:pullback_special_case}}\cong(\pi^*\ell^1)^*\\
&\overset{\phantom{\eqref{eq:IDL_vs_dual}}}\cong\big(\pi^*\varinjlim\B_\star\big)^*
\overset{\eqref{eq:DL_vs_pullback}}\cong\big(\varinjlim \pi^*\B_\star\big)^*
\overset{\eqref{eq:IDL_vs_dual}}\cong\varprojlim(\pi^*\B_\star)^*
\overset{\eqref{eq:pullback_special_case}}\cong\varprojlim L^0\big([0,1],\B_\star\big)^*\\
&\overset{\eqref{eq:RNP_ex}}\cong\varprojlim L^0\big([0,1],\B'_\star\big)
\overset{\eqref{eq:pullback_special_case}}\cong\varprojlim\pi^*\B'_\star.
\end{split}\]
Summing up, we have found an inverse system
\(\big(\{\B'_i\}_{i\in I},\{\varphi^{\rm adj}_{ij}\}_{i\leq j}\big)\)
of Banach spaces for which it holds that
\(\pi^*\varprojlim\B'_\star\ncong\varprojlim\pi^*\B'_\star\).
\fr}\end{remark}
\def\cprime{$'$} \def\cprime{$'$}

\end{document}